\documentclass[final,review,onefignum,onetabnum]{siamart171218}

\ifpdf
\hypersetup{
  pdftitle={Limited memory Kelley's Method Converges for Composite Convex and Submodular Objectives},
  pdfauthor={Song Zhou, Swati Gupta, and Madeleine Udell}
}
\fi

\usepackage[utf8]{inputenc} 
\usepackage[T1]{fontenc}    
\usepackage{booktabs}       
\usepackage{amsfonts}       
\usepackage{nicefrac}       
\usepackage{microtype}      
\usepackage{graphicx,fullpage, tabularx}
\usepackage{comment}
\usepackage{algorithm}
\usepackage[noend]{algpseudocode}
\usepackage{cite}
\usepackage{enumitem}
\usepackage{mathtools}
\usepackage{multirow}
\usepackage{subcaption}
\usepackage{graphicx}
 
\newcommand{\A}{\mathcal A} 
\newcommand{\V}{\mathcal V} 
\newtheorem*{remark}{Remark}
\DeclareMathOperator*{\argmin}{argmin}
\DeclareMathOperator*{\argmax}{argmax}
\DeclareMathOperator*{\conv}{conv}

\newcommand{\journal}[1]{}

\newcommand{\oldproof}{0} 

\newcommand{\reals}{\mathbb{R}}
\newcommand{\minimize}{\mathop{\textup{minimize}}}
\newcommand{\maximize}{\mathop{\textup{maximize}}}
\newcommand{\st}{\mathop{\textup{subject to}}}
\newcommand{\bopt}{\begin{array}{ll}}
\newcommand{\eopt}{\end{array}}
\newcommand{\beas}{\begin{eqnarray*}}
\newcommand{\eeas}{\end{eqnarray*}}
\newcommand{\bea}{\begin{eqnarray}}
\newcommand{\eea}{\end{eqnarray}}
\newcommand{\beq}{\begin{equation}}
\newcommand{\eeq}{\end{equation}}
\newcommand{\bit}{\begin{itemize}}
\newcommand{\eit}{\end{itemize}}
\newcommand{\ben}{\begin{enumerate}}
\newcommand{\een}{\end{enumerate}}

\title{Limited memory Kelley's Method Converges for Composite Convex and Submodular Objectives}
\author{Song Zhou\thanks{Cornell University, Ithaca, NY (\email{sz557@cornell.edu}).}
\and Swati Gupta\thanks{Georgia Institute of Technology, Atlanta, GA (\email{swatig@gatech.edu}).}
\and Madeleine Udell\thanks{Cornell University, Ithaca, NY (\email{udell@cornell.edu}).}}

\begin{document}

\maketitle

\begin{abstract}
The original simplicial method ({\sc OSM}), a variant of the classic Kelley's cutting plane method,
has been shown to converge to the minimizer of a composite convex and submodular objective,
though no rate of convergence for this method was known. Moreover, {\sc OSM} is required to solve subproblems in each iteration
whose size grows linearly in the number of iterations. We propose a limited memory version of Kelley's method ({\sc L-KM}) and of {\sc OSM} that requires
limited memory (at most $n+1$ constraints for an $n$-dimensional problem) independent of the iteration. We prove convergence for {\sc L-KM} when the convex part of the objective ($g$) is strongly convex
and show it converges linearly when $g$ is also smooth. Our analysis relies on duality between minimization of the composite objective and minimization of a convex function over the corresponding submodular base polytope. We introduce a limited memory version, {\sc L-FCFW},
of the Fully-Corrective Frank-Wolfe ({\sc FCFW}) method with approximate correction,
to solve the dual problem. We show that {\sc L-FCFW} and {\sc L-KM} are dual algorithms
that produce the same sequence of iterates;
hence both converge linearly (when $g$ is smooth and strongly convex) and with limited memory. We propose {\sc L-KM} to minimize composite convex and submodular objectives;
however, our results on {\sc L-FCFW} hold for general polytopes and may be of independent interest.
\end{abstract}



\section{Introduction} One of the earliest and fundamental methods to minimize non-smooth convex objectives is Kelley's method, which minimizes the maximum of lower bounds on the convex function given by the supporting hyperplanes to the function at each  previously queried point. An approximate solution to the minimization problem is found by minimizing this piecewise linear approximation, and the approximation is then strengthened by adding the supporting hyperplane at the current approximate solution \cite{Kelley1960,Cheney1959}. Many variants of Kelley's method have been analyzed in the literature \cite[for e.g.]{Lemarechal1995, Kiwiel1995, Drori2016}. Kelley's method and its variants are a natural fit for problem involving a piecewise linear function, such as composite convex and submodular objectives. This paper defines a new limited memory version of Kelley's method adapted to composite convex and submodular objectives, and establishes the first convergence rate for such a method, solving the open problem proposed in \cite{Bach2015,Bach2013}.

Submodularity is a discrete analogue of convexity and has been used to model combinatorial constraints in a wide variety of machine learning applications, such as MAP inference, document summarization, sensor placement, clustering, image segmentation \cite[and references therein]{Bach2013}. Submodular set functions are defined with respect to a ground set of elements $V$, which we may identify with $\{1,\ldots,n\}$ where $|V|=n$. These functions capture the property of diminishing returns: \(F: \{0,1\}^n \rightarrow \mathbb{R}\) is said to be submodular if \(F(A\cup \{e\}) - F(A) \geq F(B \cup \{e\}) - F(B)\) for all \(A \subseteq B \subseteq V\), $e \notin B$. Lov\'{a}sz gave a convex extension $f: [0,1]^n \rightarrow \mathbb{R}$ of the submodular set functions which takes the value of the set function on the vertices of the $[0,1]^n$ hypercube, i.e. $f(\mathbf{1}_{S}) = F(S)$, where $\mathbf{1}_S$ is the indicator vector of the set $S \subseteq V$ \cite{Lovasz1983}. (See Section \ref{sec:background} for details.)

In this work, we propose a variant of Kelley's method,
{\sc Limited Memory Kelley's Method} ({\sc L-KM}),
to minimize the composite convex and submodular objective
\begin{equation}
\label{eq:primal}
\tag{$P$}
\minimize \quad g(x) + f(x),
\end{equation}
where \(g: \mathbb{R}^n \rightarrow \mathbb{R}\) is a closed strongly convex proper function and \(f: \mathbb{R}^n \rightarrow \mathbb{R}\) is the Lov\'{a}sz extension (see Section \ref{sec:background} for details) of a given submodular function \(F: 2^{|E|} \rightarrow \mathbb{R}\).
Composite convex and submodular objectives have been used extensively in sparse learning, where the support of the model must satisfy certain combinatorial constraints.
{\sc L-KM} builds on the {\sc Original Simplicial Method} ({\sc OSM}),
proposed by Bach \cite{Bach2013} to minimize such composite objectives.
At the $i$th iteration, {\sc OSM} approximates the Lov\'{a}sz extension by a piecewise linear function $f_{(i)}$
whose epigraph is the maximum of the supporting hyperplanes to the function at each  previously queried point.
It is natural to approximate the submodular part of the objective by a piecewise linear function,
since the Lov\'{a}sz extension is piecewise linear (with possibly an exponential number of pieces).
{\sc OSM} terminates once the algorithm reaches the optimal solution,
in contrast to a subgradient method, which might continue to oscillate.
Contrast {\sc OSM} with Kelley's Method:
Kelley's Method approximates the full objective function using a piecewise linear function,
while {\sc OSM} only approximates the Lov\'{a}sz extension \(f\) and uses the exact form of $g$.
In \cite{Bach2013}, the authors show that {\sc OSM} converges to the optimum; however, no rate of convergence is given. Moreover, {\sc OSM} maintains an approximation of the Lov\'{a}sz extension by maintaining a set of linear constraints whose size grows linearly with the number of iterations. Hence the subproblems are harder to solve with each iteration.

This paper introduces {\sc L-KM}, a variant of {\sc OSM} that
uses no more than $n+1$ linear constraints in each approximation $f_{(i)}$ (and often, fewer)
and provably converges when $g$ is strongly convex.
When in addition $g$ is smooth, our new analysis of {\sc L-KM} shows that it converges linearly, and,
as a corollary, that {\sc OSM} also converges linearly, which was previously unknown.

To establish this result, we introduce the algorithm {\sc L-FCFW} to solve a problem dual to \eqref{eq:primal}:
\begin{equation}
\label{eq:dual}
\tag{$D$}
\bopt
\maximize & h(w) \\
\st & w \in B(F),
\eopt
\end{equation}
where $h: \mathbb{R}^n \to \mathbb{R}$ is a smooth concave function
and $B(F) \subset \mathbb{R}^n$ is the submodular base polytope corresponding to a given submodular function $F$ (defined below).
We show {\sc L-FCFW} is a limited memory version of the {\sc Fully-Corrective Frank-Wolfe} ({\sc FCFW}) method with approximate correction \cite{Lacoste2015},
and hence converges linearly to a solution of \eqref{eq:dual}.

We show that {\sc L-KM} and {\sc L-FCFW} are \emph{dual algorithms} in the sense that
both algorithms produce the same sequence of primal iterates and lower and upper bounds on the objective.
This connection
immediately implies that {\sc L-KM} converges linearly.
Furthermore, when $g$ is smooth as well as strongly convex,
we can recover the dual iterates of {\sc L-FCFW} from the primal iterates of {\sc L-KM}.




\textbf{Related Work:} The Original Simplicial Method was proposed by Bach (2013) \cite{Bach2013}.
As mentioned earlier, it converges finitely but no known rate of convergence was known before the present work.
In 2015, Lacoste-Julien and Jaggi proved global linear convergence of variants of the Frank-Wolfe algorithm,
including the Fully Corrective Frank-Wolfe ({\sc FCFW}) with approximate correction \cite{Lacoste2015}.
{\sc L-FCFW}, proposed in this paper, can be shown to be a limited memory special case of the latter,
which proves linear convergence of both {\sc L-KM} and {\sc OSM}.

Many authors have studied convergence guarantees and reduced memory requirements for variants of Kelley's method \cite{Kelley1960,Cheney1959}.
These variants are computationally disadvantaged compared to {\sc OSM}
unless these variants allow approximation of only part of the objective.
Among the earliest work on bounded storage in proximal level bundle methods is a paper by Kiwiel (1995) \cite{Kiwiel1995}.
This method projects iterates onto successive approximations of the level set of the objective;
however, unlike our method, it is sensitive to the choice of parameters (level sets)
and oblivious to any simplicial structure:
iterates are often not extreme points of the epigraph of the function.
Subsequent work on the proximal setup uses  trust regions, penalty functions, level sets,
and other more complex algorithmic tools;
we refer the reader to \cite{Makela2002} for a survey on bundle methods.
For the dual problem, a paper by Von Hohenbalken (1977) \cite{Von1977} shares some elements of our proof techniques. However, their results only apply to differentiable objectives and do not bound the memory. Another restricted simplicial decomposition method was proposed by Hearn et. al. (1987) \cite{Hearn1987},
which limits the constraint set by user-defined parameters
(e.g., $r=1$ reduces to the Frank-Wolfe algorithm \cite{FrankWolfe1956}):
it can replace an atom with minimal weight in the current convex combination with a prior iterate of the algorithm, which may be strictly inside the feasible region. In contrast, {\sc L-FCFW} obeys a known upper bound ($n+1$) on the number of vertices, and hence requires no parameter tuning. 

\textbf{Applications:} Composite convex and submodular objectives
have gained popularity over the last few years in a large number of machine learning applications such as structured regularization or empirical risk minimization \cite{Bach2010}: $\min_{w\in \mathbb{R}^n} \sum_{i} l(y_i, w^{\top}x_i) + \lambda \Omega(w)$, where $w$ are the model parameters and $\Omega: \mathbb{R}^n \to \mathbb{R}$ is a regularizer.
The Lov\'{a}sz extension can be used to obtain a convex relaxation of a regularizer that penalizes the support of the solution $w$ to achieve structured sparsity, which improves model interpretable or encodes knowledge about the domain. For example, fused regularization uses $\Omega(w) = \sum_{i}|w_i-w_{i+1}|$, which is the Lov\'{a}sz extension of the generalized cut function, and group regularization uses $\Omega(w) = \sum_{g} d_g \|w_g\|_{\infty}$, which is the Lov\'{a}sz extension of the coverage submodular function. (See Appendix \ref{app:background}, Table \ref{tab:listofproblems} for details on these and other submodular functions.)

Furthermore, minimizing a composite convex and submodular objective is dual to minimizing a convex objective over a submodular polytope (under mild conditions).
This duality is central to the present work.
First-order projection-based methods like online stochastic mirror descent and its variants
require computing a Bregman projection $\min\{\omega(x) + \nabla \omega(y)^{\top}(x-y): x \in P\}$
to minimize a strictly convex function $\omega: \reals^n \to \reals$ over the set $P \subseteq \mathbb{R}^n$.
Computing this projection is often difficult, and prevents practical application of these methods,
though this class of algorithms is known to obtain near optimal convergence guarantees in various settings \cite{Nemirovski1983,Audibert2013}.
Using {\sc L-FCFW} to compute these projections
can reduce the memory requirements in variants of online mirror descent used for learning over spanning trees to reduce communication delays in networks, \cite{Koolen2010}), permutations to model scheduling delays \cite{Yasutake2011}, and k-sets for principal component analysis \cite{Warmuth2006}, to give a few examples of submodular online learning problems. Other example applications of convex minimization over submodular polytopes include computation of densest subgraphs \cite{Nagano2011size}, computation of a lower bound for the partition function of log-submodular distributions \cite{Djolonga2014} and distributed routing \cite{Krichene2015convergence}.

\textbf{Summary of contributions:} We discuss  background and the problem formulations in  \cref{sec:background}.
\cref{sec:lkm} describes {\sc L-KM}, our proposed limited memory version of {\sc OSM},
and shows that {\sc L-KM} converges and solves a problem over $\reals^n$
using subproblems with at most $n+1$ constraints.
We introduce duality between our primal and dual problems in \cref{s:duality}.
\cref{s:solving-dual} introduces a limited memory (and hence faster)
version of Fully-Corrective Frank-Wolfe, {\sc L-FCFW},
and proves linear convergence of {\sc L-FCFW}.
We establish the duality between {\sc L-KM} and {\sc L-FCFW} in \cref{s:primal-from-dual} and
thereby show {\sc L-KM} achieves linear convergence and
{\sc L-FCFW} solves subproblems over no more than $n+1$ vertices.
We present preliminary experiments in Section \ref{sec:experiments} that highlight the reduced memory usage
of both {\sc L-KM} and {\sc L-FCFW} and show that their performance compares favorably with
{\sc OSM} and {\sc FCFW} respectively.

\section{Background and Notation}\label{sec:background}
Consider a ground set \(V\) of $n$ elements on which the submodular function $F:2^V \rightarrow \mathbb{R}$ is defined. The function $F$ is said to be submodular if \(F(A) \mathop{+} F(B) \geq F(A \mathop{\cup} B) \mathop{+} F(A \mathop{\cap} B)\) for all \(A,\ B \subseteq V\).
This is equivalent to the diminishing marginal returns characterization mentioned before. Without loss of generality, we assume \(F(\emptyset) = 0\).
For \(x \in \mathbb{R}^n\), \(A \subseteq V\), we define \(x(A) = \sum_{k \mathop{\in} A} x(k) = \mathbf{1}_{A}^{\top}x \), where \(\mathbf{1}_{A} \in \mathbb{R}^n\) is the indicator vector of \(A\), and let both \(x(k)\) and \(x_k\) denote the $k$th element of \(x\).

Given a submodular set function $F: V \rightarrow \mathbb{R}$, the submodular polyhedron
and the base polytope are defined as \(P(F) = \{w \in \mathbb{R}^n : w(A) \leq F(A),\ \mathop{\mathop{\forall}} A \subseteq V\},\) and
\(B(F) = \{w \in \mathbb{R}^n : w(V) = F(V), w \in P(F)\}\), respectively.
We use $\text{vert}(B(F))$ to denote the vertex set of $B(F)$.
The Lov\'{a}sz extension of $F$
is the piecewise linear function \cite{Lovasz1983}
\begin{equation}
f(x) = \max_{w \mathop{\in} B(F)} w^{\top} x.
\label{LPforLE}
\end{equation}
The Lov\'{a}sz extension can be computed using Edmonds' greedy algorithm for maximizing linear functions over the base polytope (in $O(n \log n + n \gamma)$ time, where $\gamma$ is the time required to compute the submodular function value). This extension can be defined for any set function, however it is convex if and only if the set function is submodular \cite{Lovasz1983}. We call a permutation $\pi$ over $[n]$ \emph{consistent}\footnote{Therefore, the Lov\'{a}sz extension can also be written as
$f(x) = \sum_{k} x_{\pi_k}[F(\{\pi_1, \pi_2, \dots, \pi_k\}) - F(\{\pi_1, \pi_2, \dots, \pi_{k-1}\})]$ where $\pi$ is a permutation consistent with $x$ and $F(\emptyset)=0$ by assumption.}
with $x\in \mathbb{R}^n$ if $x_{\pi_i} \geq x_{\pi_j}$ whenever $i \leq j$.
Each permutation \(\pi\) corresponds to an extreme point
$x_{\pi_k} = F(\{\pi_1, \pi_2, \dots, \pi_k\}) - F(\{\pi_1, \pi_2, \dots, \pi_{k-1}\})$ of the base polytope.
For $x \in \mathbb{R}^n$, let $\mathcal{V}(x)$ be the set of vertices $B(F)$ that correspond to permutations consistent with $x$.

Note that \begin{equation}
\label{lemma:lovasz_attainment}
    \partial f(x) = \mathop{\mathop{\mathrm{conv}}}(\mathcal{V}(x)) = \argmax_{w \mathop{\in} B(F)}w^{\top}x,
\end{equation}
where \(\partial f(x)\) is the subdifferential of \(f\) at \(x\) and conv$(S)$ represents the convex hull of the set $S$.

We assume all convex functions in this paper are closed and proper \cite{Ryu2016}.
Given a convex function \(g:\mathbb{R}^{n} \rightarrow \mathbb{R}\), its Fenchel conjugate \(g^{*}:\mathbb{R}^{n} \rightarrow \mathbb{R}\) is defined as
\begin{equation}
g^{*}(w) \overset{\Delta}{=} \max\limits_{x \mathop{\in} \mathbb{R}^{n}} w^{\top}x - g(x). \label{fenchel}
\end{equation}
Note that when \(g\) is strongly convex, the right hand side of \eqref{fenchel} always has an unique solution, so \(g^*\) is defined for all \(w \in \mathbb{R}^n\). Fenchel conjugates are always convex, regardless of the convexity of the original function. Since we assume \(g\) is closed, \({g^{**}} = g\). Fenchel conjugates satisfy
$(\partial g)^{-1} = \partial g^*$ in the following sense:
\begin{equation}
w \in \partial g(x) \iff g(x) + g^{*}(w) = w^{\top} x \iff x \in \partial g^*(w),
\label{lemma:Fenchelpair}
\end{equation}
where $\partial g(x)$ is the subdifferential of $g$ at \(x\).
When $g$ is $\alpha$-strongly convex and $\beta$-smooth, $g^*$ is $1/\beta$-strongly convex and $1/\alpha$-smooth \cite[Section 4.2]{Ryu2016}. (See Appendix \ref{app:convexity} for details.)

Proofs of all results that do not follow easily from the main text can be found in the appendix.

\if \oldproof
\textbf{Problem Formulation:} We consider a primal-dual pair of convex minimization problems involving a given submodular function \(F: 2^{V} \to \mathbb{R}\). We refer to the first problem as the primal problem and the second as the dual problem.
\journal{
    \footnote{Although \cite{Bach2013} mentions that strong duality holds between the primal and the dual forms, Bach does not discuss attainability of the solutions. 
    We include a proof here for completeness.)}
}
The first problem is to minimize a composite convex and submodular objective:
\begin{equation}
\tag{$P$} 
\minimize~g(x) +f (x) \qquad \mbox{subject to}~x \in \mathbb{R}^n
\label{primal}
\end{equation}
where \(g: \mathbb{R}^n \rightarrow \mathbb{R}\) is a closed\footnote{A function \(g: \mathbb{R}^{n} \rightarrow \mathbb{R}\) is said to be closed iff \(\forall \alpha \mathop{\in} \mathbb{R}\), \{\(x \in \mathop{\mathrm{dom}}(g): f(x) \leq \alpha\)\} is a closed set.} \SZ{strongly convex strongly smooth} convex function, and \(f: \mathbb{R}^n \rightarrow \mathbb{R}\) is the Lov\'{a}sz extension of \(F\). Since indicator functions of convex sets are also convex functions, the primal problem can model constrained minimization problems by setting \(g\) to be the corresponding indicator functions. 
The dual problem is to maximize a concave function over the submodular base polytope $B(F)$:
\begin{equation}
\tag{$D$} 
\maximize~h(w) \qquad \mbox{subject to}~w \in B(F)
\label{dual}
\end{equation}
where \(h: \mathbb{R}^n \rightarrow \mathbb{R}\) is a smooth concave function. 
We have assumed that $g$ and $h$ are both strongly smooth and strongly convex
so that, for any $F$, both the primal and dual problems have a unique solution that is attained.
Solving the primal form \eqref{eq:primal} using general convex optimization methods is challenging due to the non-differentiability of \(f\) at any point $x$ that has at least two elements of the same value. On the other hand, in the case of the dual form, the number of facets of \(B(F)\) can be exponential in the dimension of the polytope\footnote{In fact, there exists a family of matroids defined over ground set of size $n$, such that the extension complexity of convex hull of the bases i.e. the corresponding base polytope is $\Omega(2^{n/2}/n^{5/4 \sqrt{\log(2n)}})$ \cite{Rothvoss2013}.}. 

In \cite{Bach2013}, the duality between (\eqref{eq:primal}) and (\eqref{eq:dual}) was established and studied. 
Here we show strong duality using a standard theorem in convex analysis, 
in contrast to the direct approach of \cite{Bach2013}:
\begin{theorem}[Strong Duality (Theorem 3.3.5, \cite{Borwein2010})]
\label{strong_duality}
Let \(f\) be the Lov\'{a}sz extension of a submodular function \(F: 2^{V} \rightarrow \mathbb{R}\) with \(|V| = n\), let \(g: \mathbb{R}^{n} \rightarrow \mathbb{R}\) be a convex function whose domain has non-empty interior. We have \(p^{\star} = \min_{x \mathop{\in} \mathbb{R}^n} g(x) + f(x) = d^{\star}\), where
\begin{equation}
\tag{$D'$}
\label{dual_new_form_old}
d^{\star} = \maximize~-g^{*}(-w) \qquad \mbox{subject to}~w \in B(F).
\end{equation}
Moreover, both the primal problem \eqref{eq:primal} and the dual problem \eqref{dual_new_form} have finite optimal solutions.
\end{theorem}

\SZZ{begins}

\begin{theorem}[Strong Duality (Theorem 3.3.5, \cite{Borwein2010})]
\label{strong_duality_old}
Let \(\mathcal{W} \subseteq \mathbb{R}^n\) be a finite set, let \(g: \mathbb{R}^{n} \rightarrow \mathbb{R}\) be a \SZ{closed strongly} convex function. We have \(p^{\star} = d^{\star}\), where 
\begin{equation}
\label{primal_gen_form}
\tag{$P_{\mathcal W}$}
p^{\star} = \min_{x \mathop{\in} \mathbb{R}^n}\{ g(x) + \max_{w \mathop{\in} \mathcal{W}} x^{\top} w\}
\end{equation}
and
\begin{equation}
\tag{$D_{\mathcal W}$}
\label{dual_gen_form}
d^{\star} = \maximize~-g^{*}(-w) \qquad \mbox{subject to}~w \in \conv(\mathcal{W}).
\end{equation}
Moreover, both the primal problem \eqref{primal_gen_form} and the dual problem \eqref{dual_new_form} have finite optimal solutions.
\end{theorem}

Let \(\mathcal{W} = B(F)\), we have strong duality between  \eqref{eq:primal} and

\begin{equation}
\tag{$D'$}
\label{dual_new_form}
\maximize~-g^{*}(-w) \qquad \mbox{subject to}~w \in B(F).
\end{equation}

Note that problems \eqref{eq:dual} and \eqref{dual_new_form} have the same form, we refer to both as the dual problem. We now present some results regarding the relationship between the optimal solutions of \eqref{eq:primal} and \eqref{dual2}:


\begin{lemma}
[Primal solution recovers dual solution]
\label{primal_to_dual}
Suppose \(g\) is a \SZ{closed strongly} convex function and and $x^\star$ solves the primal problem \eqref{eq:primal}. Then \(w^\star = -\nabla g(x^\star)\) solves the dual problem \eqref{dual_new_form}.
\end{lemma}

\SZZ{begins}

\begin{proof}
We know \(w^{\star}\) is unique from the fact that \(g\) is strongly convex. By the optimality of \(x^{\star}\), we have \(0 \in \partial g (x^{\star}) + \partial f (x^{\star}) =- w^{\star} + \partial f (x^{\star})\). Thus \(w^{\star} \in \partial f(x^{\star})\). and \(f(x^{\star}) = {w^{\star}}^{\top}x^{\star}\) from \eqref{lemma:lovasz_attainment}. Hence Using \eqref{Fenchelpair} from preliminaries, we get
\(-g^{*}(-w^{\star}) \mathop{=} g(x^{\star}) +  {w^{\star}}^{\top}x^{\star} 
\mathop{=} g(x^{\star}) + f(x^{\star})\). Using strong duality, we get \(w^{\star}\) solves the dual problem \eqref{dual_new_form}.
\end{proof}

\SZZ{ends}

\begin{lemma}[Dual solution recovers primal solution]
\label{dual_to_primal} 
Suppose \(g\) is a \SZ{closed strongly convex strongly smooth} function and \(w^{\star} \in B(F)\) solves the dual problem \eqref{dual_new_form}. Then \(x^{\star} = \nabla g^{*}(-w^{\star})\) solves the primal problem \eqref{eq:primal}.
\end{lemma}

\SZZ{begins}

\begin{proof}
Since \(x^{\star} = \nabla g^{*}(-w^{\star})\) and \(w^{\star}\) solves the dual problem, we have \(w^{\star}\) is optimal to \(\argmax_{w \mathop{\in} B(F)}(-w)^{\top}(-x^{\star}) = \argmax_{w \mathop{\in} B(F)}w^{\top}x^{\star}\), and \(f(x^{\star}) = {w^{\star}}^{\top}x^{\star}\). Note that \(g\) is closed and \(-x^{\star} = -g^{*}(-w^{\star})\), we have \(g(x^{\star}) = g^{**}(x^{\star}) = -g^{*}(-w^{\star}) + {w^{\star}}^{\top}(-x^{\star}) = -g^{*}(-w^{\star}) - f(x^{\star})\) from \eqref{Fenchelpair}. Thus \(x^{\star}\) solves the primal problem \eqref{eq:primal} from strong duality.
\end{proof}



\SZZ{ends}

\mnote{Prove these.}

\mnote{No need for weak duality, since we already have strong. Delete (or, save for journal version).}
As shown in \cite{Bach2013}, using the definitions of the Lov\'{a}sz extension and the Fenchel conjugates, it is easy to see that weak duality always holds between problem \eqref{eq:primal} and \eqref{eq:dual}. We include a proof of \cref{weakduality} in Appendix \ref{app:duality} for completeness.

\begin{lemma}[Weak Duality] \label{weakduality}
Let \(f\) be the Lov\'{a}sz extension of a submodular function \(F: 2^{V} \rightarrow \mathbb{R}\) with \(|V| = n\) and let \(g: \mathbb{R}^n \rightarrow \mathbb{R}\) be a convex function. We have \(p^{\star} = \minimize_{x \mathop{\in} \mathbb{R}^n} g(x) + f(x) \geq d^{\star}\), where
\begin{equation}
\tag{$D^\prime$}
\label{dual2}
d^{\star} = \maximize~-g^{*}(-w) \qquad \mbox{subject to}~w \in B(F)
\end{equation}
\end{lemma}
Note that problems $(D)$ and $(D^\prime)$ have the same form.
We refer to both as the dual problem. To argue strong duality, we require the following lemma that characterizes optimality conditions in terms of the intersection of the sub-differentials of $f$ and $-g$.

\begin{lemma}\label{lemma2} 
Let $X^{\star} = \argmin_{x \in \mathbb{R}^n} g(x) + f(x)$ and 
$W^{\star} = \arg\max_{w \in B(F)} -g^*(-w)$ 
be the set of optimal solutions to the primal and dual problems respectively.
When \(X^{\star}\) is nonempty, for any $x^{\star} \in X^{\star}$, 
the negative of $\partial g(x^{\star})$ has a non-empty intersection with 
$\partial f(x^{\star})$: 
$-\partial g(x^{\star}) \cap \partial f(x^{\star}) \neq \emptyset$. 
Moreover, any vector in this intersection is a solution to the dual problem: $-\partial g(x^{\star}) \cap \partial f(x^{\star}) \subseteq W^{\star}$.
\end{lemma}
The proof of \cref{lemma2} follows from \eqref{Fenchelpair} and weak duality, and appears in Appendix \ref{app:duality}. It further follows from \cref{lemma2} that strong duality holds whenever there exists an optimal solution to the primal or the dual problem, as $\emptyset \neq -\partial g(x^\star) \cap \partial f(x^\star)\subseteq W^\star$, and thus $W^\star$ is also attained. 
\begin{lemma}[Strong Duality]\label{strongduality}
Whenever a solution $x^\star \in \arg\min_{x} g(x) + f(x)$ exists, then some $w^\star \in \arg\max_{w \in B(F)} -g^*(-w)$ also exists and strong duality holds: $g(x^\star) + f(x^\star) = -g^*(-w^\star)$. 
\end{lemma}
As a corollary of \cref{lemma2}, one can construct a solution to the dual problem given a solution to the primal problem at which $g$ or $f$ is differentiable (see Corollary \eqref{dualsolution}, Appendix \ref{app:duality}).
\fi

\section{Limited Memory Kelley's Method}\label{sec:lkm}
We now present our novel limited memory adaptation {\sc L-KM} of the Original Simplicial Method ({\sc OSM}). We first briefly review {\sc OSM} as proposed by Bach \cite[Section 7.7]{Bach2013} and discuss problems of {\sc OSM} with respect to memory requirements and the rate of convergence. We then highlight the changes in {\sc OSM}, and verify that these changes will enable us to show a bound on the memory requirements while maintaining finite convergence.
Proofs omitted from this section can be found in Appendix \ref{app:lkm}.

\textbf{Original Simplicial Method:} To solve the primal problem
\eqref{eq:primal}, it is natural to approximate the piecewise linear Lov\'{a}sz extension $f$ with cutting planes derived from the function values and subgradients of the function at previous iterations, which results in piecewise linear lower approximations of \(f\). This is the basic idea of {\sc OSM} introduced by Bach in \cite{Bach2013}. This approach contrasts with Kelley's Method, which approximates the entire objective function $g+f$.
{\sc OSM} adds a cutting plane to the approximation of \(f\) at each iteration, so the number of the linear constraints in its subproblems grows linearly with the number of iterations.\footnote{
Concretely, we obtain {\sc OSM} from \cref{alg:L-KM} by setting $\mathcal{V}^{(i)} \overset{\Delta}{=} \V^{(i-1)}\cup \{v^{(i)}\}$ in step \ref{alg:memory_update}.
}
Hence it becomes increasingly challenging to solve the subproblem as the number of iterations grows up.
Further, in spite of a finite convergence, as mentioned in the introduction there was no known rate of convergence for {\sc OSM} or its dual method prior to this work.

\textbf{Limited Memory Kelley's Method:} To address these two challenges
--- memory requirements and unknown convergence rate ---
we introduce and analyze a novel limited memory version {\sc L-KM} of {\sc OSM} which ensures that the number of cutting planes maintained by the algorithm at any iteration is bounded by $n+1$.
This thrift bounds the size of the subproblems at any iteration, thereby making {\sc L-KM} cheaper to implement. We describe {\sc L-KM} in detail in \cref{alg:L-KM}.


\begin{algorithm}[!t]\footnotesize
\caption{{\sc L-KM}: The Limited Memory Kelley's Method for \eqref{eq:primal}
\label{alg:L-KM}}
\begin{algorithmic}[1]
\Require strongly convex function \(g: \mathbb{R}^n \rightarrow \mathbb{R}\), 
submodular function \(F: 2^n \rightarrow \mathbb{R}\),
tolerance \(\epsilon \geq 0\)
\Ensure $\epsilon$-suboptimal solution \(x^{\sharp}\) to \eqref{eq:primal}
\State initialize:
choose \(x^{(0)} \in \mathbb{R}^{n}\), 
set \(\emptyset \subset \mathcal{V}^{(0)} \subseteq \mathrm{vert}(B(F))\) affinely independent
\For {$i=1,2,\ldots$}
 \State \textbf{Convex subproblem.} Define approximation \(f_{(i)}(x) = \max\{w^{\top}x : w \mathop{\in} \mathcal{V}^{(i-1)}\}\) and solve
 \[
 x^{(i)} = \argmin g(x) + f_{(i)}(x).    \label{subp}
 \]
 \State \textbf{Submodular subproblem}. Compute value and subgradient of $f$ at $x^{(i)}$ 
 \[
 f(x^{(i)}) = \max_{w \in B(F)} w^{\top} x^{(i)}, \qquad v^{(i)} \mathop{\in} \partial f(x^{(i)}) = \argmax_{w \in B(F)} w^{\top} x^{(i)}.
 \]
\State \textbf{Stopping condition.} Break if duality gap \(p^{(i)} - d^{(i)} \leq \epsilon\), where
    \[
    p^{(i)} = g(x^{(i)}) + f(x^{(i)}), \qquad d^{(i)} = g(x^{(i)}) + f_{(i)}(x^{(i)}).
    \]
  \State \textbf{Update memory.} Identify active set \(\mathcal{A}^{(i)}\) and update memory \(\mathcal{V}^{(i)}\):
 \[
 \mathcal{A}^{(i)} = \{w \in \mathcal{V}^{(i-1)}: w^{\top}{x^{(i)}} = f_{(i)}(x^{(i)})\},
 \qquad
 \mathcal{V}^{(i)} = \mathcal{A}^{(i)} \cup \{v^{(i)}\}.
 \] \label{alg:memory_update}
\EndFor
\State \Return \(x^{(i)}\)

\end{algorithmic}
\end{algorithm} 

\journal{We use \(z^{(i)}\) and \(p^{(i)}\) to keep track of the best iterate and its value so far. The current iterate $x^{(i)}$ can be worse in terms of the function value than previous iterates (e.g. \cref{lmfig} (b), (c)).}


\begin{figure}
\centering
 \begin{subfigure}{3.3cm}
  \centering\includegraphics[width=3.3cm]{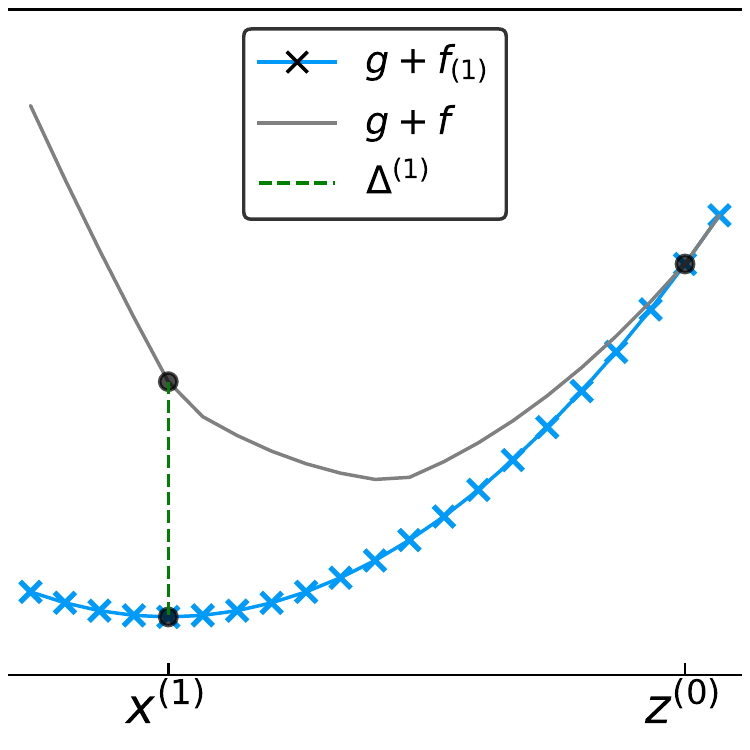}
 \end{subfigure}
 \begin{subfigure}{3.3cm}
   \centering\includegraphics[width=3.3cm]{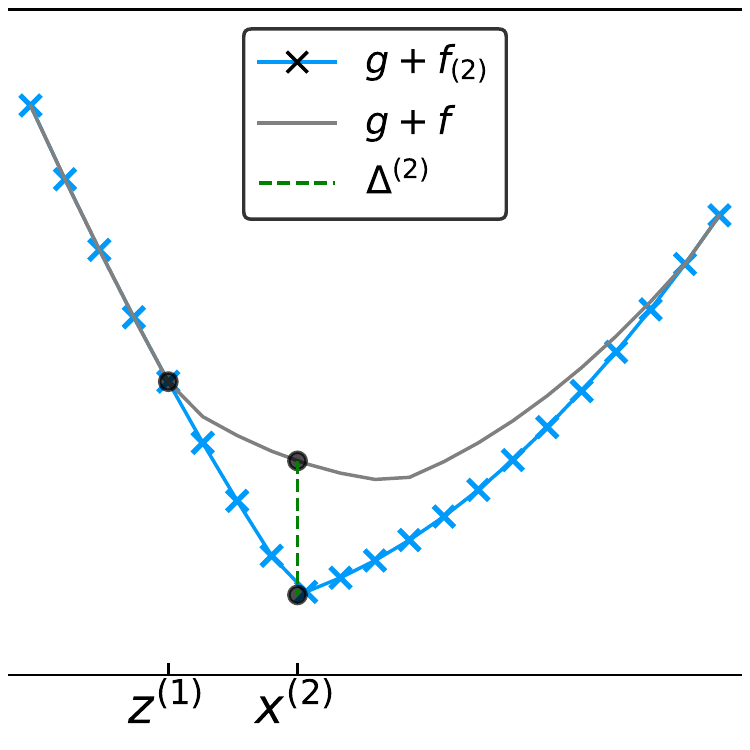}
 \end{subfigure}
 \begin{subfigure}{3.3cm}
   \centering\includegraphics[width=3.3cm]{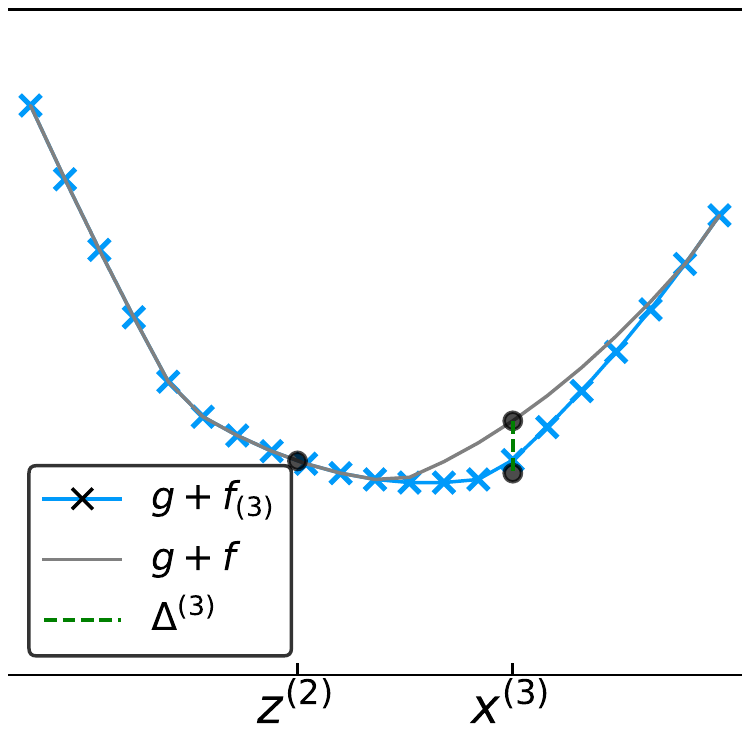}
 \end{subfigure}
 \begin{subfigure}{3.3cm}
   \centering\includegraphics[width=3.3cm]{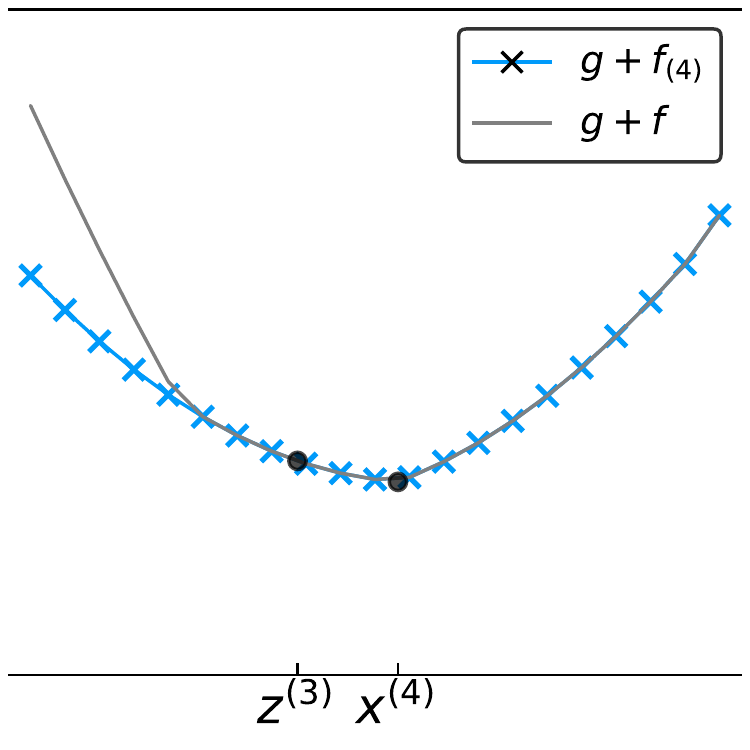}
 \end{subfigure}
 \caption{\footnotesize An illustration of {\sc L-KM} (a)-(d) (left to right): blue curve marked with \(\times\) denotes the $i$th function approximation $g + f_{(i)}$. In (d), note that {\sc L-KM} approximation \(g + f_{(4)}\) is obtained by dropping the leftmost constraint in $g+f_{(3)}$ (in (c)), unlike {\sc OSM}.}
  \label{lmfig}
\end{figure}


{\sc L-KM} and {\sc OSM} differ only in the set of vertices $\mathcal V^{(i)}$ considered at each step:
{\sc L-KM} keeps only those vectors $w \in \mathcal{V}^{(i-1)}$ that maximize $w^{\top}x^{(i)}$,
whereas {\sc OSM} keeps every vector in \(\mathcal{V}^{(i-1)}\). 

We state some properties of {\sc L-KM} here with
proofs in \cref{app:lkm}. We will revisit many of these
properties later via the lens of duality. 

The sets $\mathcal{V}^{(i)}$ in {\sc L-KM} are affinely independent, which shows the size of the subproblems is bounded.
\begin{theorem}
\label{thm:limited}
For all \(i \geq 0\), vectors in \(\mathcal{V}^{(i)}\) are affinely independent. Moreover, $|\mathcal{V}^{(i)}|\leq n+1$.
\end{theorem}

{\sc L-KM} may discard pieces of the lower approximation \(f_{(i)}\) at each iteration.
However, it does so without any adverse affect on the solution to the current subproblem:
\begin{lemma}
The convex subproblem (in Step 3 of algorithm {\sc L-KM}) has the same solution and optimal value 
over the new active set $\mathcal{A}^{(i)}$
as over the memory $\mathcal{V}^{(i-1)}$: 
\[
x^{(i)}
= \argmin g(x) + \max\{w^{\top}x : w \in \mathcal{V}^{(i-1)}\}
= \argmin g(x) + \max\{w^{\top}x : w \in \mathcal{A}^{(i)}\}.\]
\label{lemma:equality}
\end{lemma}
\cref{lemma:equality} shows that {\sc L-KM} remembers the important information about the solution,  i.e. only the tight subgradients, at each iteration. Note that at the \(i\)th iteration, the solution \(x^{(i)}\) is unique by the strong convexity of \(g\), and thus we can improve the lower bound \(d^{(i)}\) since new information (i.e. $v^{(i)}$) is added:
\begin{corollary}
\label{cor:lowerbounds}
The sequence \{\(d^{(i)}\)\} constructed by {\sc L-KM} form strictly increasing lower bounds
on the value of \eqref{eq:primal}: \(d^{(1)} < d^{(2)} < \cdots \leq p^{\star} \overset{\Delta}{=} \min_{x \mathop{\in} \mathbb{R}^n} f(x) + g(x)\).
\end{corollary}

\begin{remark}
It is easy to see that the sequence \(\{p^{(i)}\}\) constructed by {\sc L-KM} form upper bounds of \(p^{\star}\), hence by \cref{cor:lowerbounds}, \(\{p^{(i)}-d^{(i)}\}\) form valid optimality gaps for {\sc L-KM}.
\end{remark}

\begin{corollary}
\label{cor:inequality}
{\sc L-KM} does not stall: for any iterations \(i_1 \neq i_2\), we solve subproblems over a distinct set of vertices \(\mathcal{V}^{(i_1)} \neq \mathcal{V}^{(i_2)}\).
\end{corollary}
We can strengthen \cref{cor:inequality} and show {\sc L-KM} in fact converges to the
exact solution in finite iterations:
\begin{theorem}
\label{thm:termination}
{\sc L-KM} ({\cref{alg:L-KM}}) terminates after finitely many iterations. Moreover, for any given $\epsilon \geq 0$,
suppose {\sc L-KM} terminates when \(i = i_\epsilon\), then \(p^{\star} + \epsilon \geq p^{(i_\epsilon)} \geq p^{\star}\) and \(p^{\star} \geq d^{(i_\epsilon)} \geq p^{\star} - \epsilon\). In particular, when we choose \(\epsilon=0\), we have \(p^{(i_0)} = p^{\star} = d^{(i_\epsilon)}\), and \(x^{(i_\epsilon)}\) is the unique optimal solution to \eqref{eq:primal}.
\end{theorem}

In this section, we have shown that {\sc L-KM} solves a series of limited memory convex subproblems
with no more than $n+1$ linear constraints, and produces strictly increasing lower bounds that converge to the optimal value. 

\if oldproof

However, at any given iteration $i$, Lemma~\ref{lemma:equality} shows that deletion of vertices does not affect the solution $x^{(i)}$ of the subproblem. This lemma will be crucial in showing that limiting the set of vectors in $\V^{(i)}$ in each iteration does not effect the convergence of {\sc L-KM}.

The proof of this lemma follows by observing that $x^{(i)}$ remains locally optimal for $\widetilde{P}_{(i)}$ due to preservation of the outer normal cone supported at $w^{\natural} \in \V^{(i-1)}$ such that $f_{(i)}(x^{(i)}) = \max_{w \in \text{conv}(\V^{(i-1)})}w^{\top}x = w^{\natural T}x^{(i)}$. We include the detailed proof in Appendix \ref{app:lkm}.


\begin{corollary}\label{lowerbounds}
The sequence of lower bounds \{\(d^{(i)}\)\} constructed by {\sc L-KM} are non-decreasing, i.e. \(d^{(i-1)} \leq d^{(i)}\) for all \(i\geq 1\).
\end{corollary}
This is a direct result of \cref{lemma:equality} and the inclusion relation \(\mathcal{A}^{(i)} \subseteq \mathcal{V}^{(i)} \subseteq \mathcal{V} = \text{vert}(B(F))\). We now show that our proposed {\sc L-KM} will not stall at suboptimal iterates:


\begin{lemma} For all \(i \geq 0\), when \(x^{(i)}\) is suboptimal, \(x^{(i+1)} \neq x^{(i)}\).
\end{lemma}
\begin{proof}
We prove this by contradiction. If \(x^{(i)}\) is suboptimal and \(x^{(i+1)} = x^{(i)}\), we have 
\begin{equation}
d^{(i+1)} = g(x^{(i)}) + f_{(i+1)}(x^{(i)}) \overset{(a)}{\geq} g(x^{(i)}) + {v^{(i)}}^{\top}(x^{(i)}) \overset{(b)}{=} g(x^{(i)}) + f(x^{(i)}) \geq p^{\star},
\end{equation}
where (a) and (b) follows from \(v^{(i)} \in \mathcal{V}^{(i)}\) and \(v^{(i)} \in \mathcal{V}(x^{(i)})\) respectively. This forces \(x^{(i)}\) to be optimal, and we get a contradiction.
\end{proof}

We show next that the size of the subproblems stays bounded by $n+1$. This follows from the affine independence of the sets $\mathcal{V}^{(i)}$, as we show in the following lemma (proof in Appendix \ref{app:lkm}): 
\begin{lemma}\label{thm:limited}
For all \(i \geq 0\), vectors in \(\mathcal{V}^{(i)}\) are affinely independent. Moreover, $|\mathcal{V}^{(i)}|\leq n+1$.
\end{lemma}
Thus, we have shown that {\sc L-KM} solves a series of limited memory convex subproblems whose number of linear constraints do not exceed \(n+1\). {\sc L-KM} produces non-decreasing lower bounds and does not stall at suboptimal solutions. 
\fi

\section{Duality} \label{s:duality}

{\sc L-KM} solves a series of subproblems parametrized by the sets $\mathcal V \subseteq \mathop{\textup{vert}}(B(F))$:
\begin{equation}
\tag{\textit{P(V)}}
\bopt
\minimize & g(x) + t \\
\st & t \geq v^\top x, \quad v \in \mathcal V
\eopt
\label{eq:primal-subproblem}
\end{equation}
Notice that when $\mathcal V = \mathop{\textup{vert}}(B(F))$, we recover \eqref{eq:primal}.
We now analyze these subproblems via 
duality. The Lagrangian of this problem with dual variables $\lambda_v$ for $v \in \mathcal V$ is,
\[
\mathcal L(x,t,\lambda) = g(x) + t + \sum_{v \in \mathcal V} \lambda_v (v^\top x - t).
\]
The pair $((x,t), \lambda)$ are primal-dual optimal for this problem {\it iff} they satisfy the KKT conditions \cite{Optimization}:
\bit
\item \emph{Optimality.}
\[
0 \in \partial_x \mathcal L(x,t,\lambda) \implies \sum_{v \in \mathcal V} \lambda_v v \in -\partial g(x),
\qquad
0 = \frac d {dt} \mathcal L(x,t,\lambda) \implies \sum_{v \in \mathcal V} \lambda_v = 1.\\
\]
\item \emph{Primal feasibility.} $t \geq v^\top x$ for each $v \in \mathcal V$.
\item \emph{Dual feasibility.} $\lambda_v \geq 0$ for each $v \in \mathcal V$.
\item \emph{Complementary slackness.} $\lambda_v (v^\top x - t) = 0$ for each $v \in \mathcal V$.
\eit
The requirement that $\lambda$ lie in the simplex emerges naturally from the optimality conditions,
and reduces the Lagrangian to $\mathcal{L}(x, \lambda) = g(x) + (\sum_{v \in \mathcal V} \lambda_v v)^\top x$.
One can introduce the variable $w = \sum_{v \in \mathcal V} \lambda_v v \in \conv( \mathcal V)$,
which is dual feasible so long as $w \in \conv(\mathcal V)$.
We can rewrite the Lagrangian in terms of $x$ and $w \in \conv(\mathcal V)$ as
\(
L(x,w) = g(x) + w^\top x.
\)
Minimizing $L(x,w)$ over $x$, we obtain the dual problem
\begin{equation}
\tag{\textit{D(V)}}
\bopt
\maximize & -g^*(-w) \\
\st & w \in \conv(\mathcal V).
\eopt
\label{eq:dual-subproblem}
\end{equation}
Note \eqref{eq:dual} is the same as \eqref{eq:dual-subproblem} if $\mathcal V = \mathrm{vert}(B(F))$ and $h(w) = -g^*(-w)$, the Fenchel conjugate of $g$.
Notice that $g^*$ is smooth if $g$ is strongly convex
(\cref{lemma:convexity_smoothness} in Appendix \ref{app:convexity}).

\begin{theorem}[Strong Duality]
The primal problem \eqref{eq:primal-subproblem} and the dual problem \eqref{eq:dual-subproblem} have the same finite optimal value. \label{thm:strong-duality-subproblems}
\end{theorem}

By analyzing the KKT conditions, we obtain the following result,
which we will used later in the design of our algorithms.
\begin{lemma} 
\label{lemma:comp-slackness}
Suppose $(x,\lambda)$ solve (\eqref{eq:primal-subproblem}, \eqref{eq:dual-subproblem}) and $t = \max_{v \in \mathcal V} v^\top x$.
By complementary slackness,
\beas
\lambda_v > 0 & \implies & v^\top x = t, \\
\text{and in particular,}~~ \{v: \lambda_v > 0 \} & \subseteq & \{v: v^\top x = t\}.
\eeas
\end{lemma}
Notice $\{v: v^\top x = t\}$ is the active set of {\sc L-KM}.
We will see $\{v: \lambda_v > 0 \}$ is the (minimal) active set of the dual method {\sc L-FCFW}.
(If strict complementary slackness holds, these sets are the same.)

The first KKT condition shows how to move between
primal and dual optimal variables.
\begin{theorem}
\label{thm:primal-from-dual}
If $g: \mathbb{R}^n \to \mathbb{R}$ is strongly convex and $w^\star$ solves \eqref{eq:dual-subproblem}, then
\beq \label{eq:primal-from-dual}
x^\star = (\partial g)^{-1}(-w^\star) = \nabla g^*(-w^\star)
\eeq
solves \eqref{eq:primal-subproblem}.
If in addition $g$ is smooth and $x^\star$ solves \eqref{eq:primal-subproblem}, then $w^\star = \nabla g(x^\star)$ solves \eqref{eq:dual-subproblem}.
\end{theorem}

\begin{proof}
Check the optimality conditions to prove the result.
By definition, $x^\star$ satisfies the first optimality condition.
To check complementary slackness, we rewrite the condition as
\[
\lambda_v(v^\top x - t) = 0 \quad \forall v \in \mathcal V
\iff
\left(\sum_{v \in \mathcal V} \lambda_v v\right)^\top x = t
\iff
w^\top x = \max_{v \in \mathcal V} v^\top x.
\]
Notice
$(w^\star)^\top (\nabla g^*(-w^\star)) = \max_{v \in \mathcal V} v^\top (\nabla g^*(-w^\star))$ by optimality of $w^\star$, since $v - w^\star$ is a feasible direction for any $v \in \mathcal V$,
proving $x^\star = \nabla g^*(-w^\star)$ solves \eqref{eq:primal-subproblem}.

That the primal optimal variable yields a dual optimal variable via $w^\star = \nabla g(x^\star)$ follows from a similar argument together with ideas from the proof of strong duality in \cref{app:duality}.
\end{proof}

\section{Solving the dual problem} \label{s:solving-dual}

Let's return to the dual problem \eqref{eq:dual}:
maximize a smooth concave function \(h(w) = -g^*(-w)\)
over the polytope $B(F) \subseteq \mathbb{R}^n$.
Linear optimization over this polytope is easy;
hence a natural strategy is to use the Frank-Wolfe method or one of its variants
\cite{Lacoste2015}.
However, since the cost of each linear minimization is not negligible,
we will adopt a Frank-Wolfe variant that makes considerable progress at each iteration
by solving a subproblem of moderate complexity:
{\sc Limited Memory fully corrective Frank-Wolfe} ({\sc L-FCFW}, \cref{alg:fcfw}),
which at every iteration
exactly minimizes the function \(-g^*(-w)\) over the the convex hull of the current
subset of vertices $\mathcal V^{(i)}$.
Here we overload notation intentionally: when \(g\) is smooth and strongly convex, we will see that we can choose
the set of vertices \(\mathcal{V}^{(i)}\) in {\sc L-FCFW} (\cref{alg:fcfw}) so that the algorithm matches either {\sc L-KM} or {\sc OSM}
depending on the choice of \(\mathcal{B}^{(i)}\) (\cref{line:choose_B} of {\sc L-FCFW}). For details of the duality between {\sc L-KM} and {\sc L-FCFW} see \cref{sec:convergence}.

\begin{algorithm}[!t]
\footnotesize
\caption{{\sc L-FCFW}: Limited Memory Fully Corrective Frank Wolfe for \eqref{eq:dual} \label{alg:fcfw}}
\begin{algorithmic}[1]

\Require smooth concave function \(h: \mathbb{R}^n \rightarrow \mathbb{R}\),
submodular function \(F: 2^n \rightarrow \mathbb{R}\),
tolerance \(\epsilon \geq 0\)
\Ensure $\epsilon$-suboptimal solution \(w^{\sharp}\) to \eqref{eq:dual}
\State initialize:
set \(\emptyset \subset \mathcal{V}^{(0)} \subseteq \mathrm{vert}(B(F))\)
\For {$i=1,2,\ldots$}
 \State \textbf{Convex subproblem.} Solve
 \[
 w^{(i)} = \argmax \{h(w):  w \in \conv(\mathcal V^{(i-1)}) \}.
 \]
 For each $v \in \mathcal V^{(i)}$, define $\lambda_v \geq 0$ so that
 $w^{(i)} = \sum_{v \in \mathcal V^{(i)}} \lambda_v v$ and $\sum_{v \in \mathcal V^{(i)}} \lambda_v=1$.

 \State \textbf{Submodular subproblem.} Compute gradient $x^{(i)} = \nabla h(w^{(i)})$ and solve
 \[
 v^{(i)} = \argmax \{w^\top x^{(i)}: w \in B(F)\}.
 \]

\State \textbf{Stopping condition.}
Break if duality gap \(p^{(i)} - d^{(i)} \leq \epsilon\), where
    \[
    p^{(i)} = (v^{(i)})^\top x^{(i)},
    \qquad
    d^{(i)} = (w^{(i)})^\top x^{(i)}.
    \]
  \State \textbf{Update memory.} Identify a supserset of active vertices \(\mathcal{B}^{(i)}\) and update memory \(\mathcal{V}^{(i)}\):
 \[
 \mathcal{B}^{(i)} \supseteq \{w \in \mathcal{V}^{(i-1)}: \lambda_w > 0 \} \label{line:choose_B}
 \qquad
 \mathcal{V}^{(i)} = \mathcal{B}^{(i)} \cup \{v^{(i)}\}.
 \]
 \label{condition_for_B}
\EndFor
\State \Return \(w^{(i)}\)

\end{algorithmic}
\end{algorithm}

\paragraph{Limited memory.}
In {\sc L-FCFW}, we may choose any active set
$\mathcal{B}^{(i)} \supseteq \{w \in \mathcal{V}^{(i-1)}: \lambda_w > 0 \}$.
When $\mathcal{B}^{(i)} = \mathcal{V}^{(i-1)}$, we call the algorithm (vanilla) {\sc FCFW}.
When $\mathcal{B}^{(i)}$ is chosen to be small, we call the algorithm
{\sc Limited Memory FCFW} ({\sc L-FCFW}).
Standard {\sc FCFW} increases the size of the active set at each iteration,
whereas the most limited memory variant of {\sc L-FCFW} uses only those vertices
needed to represent the iterate $w^{(i)}$.

Moreover, recall Carath\'{e}odory's theorem (see e.g. \cite{vershynin2018high}): 
for any set of vectors $\mathcal{V}$, if $x \in \conv(\mathcal{V}) \subseteq \reals^n$, then
there exists a subset $A \subseteq \mathcal{V}$ with $|A|\leq n+1$ such that $x \in \conv(A)$.
Hence we see we can choose $\mathcal B^{(i)}$ to contain at most $n+1$ vertices at each iteration (hence \(n+2\) in \(\mathcal{V}^{(i)}\)),
or even fewer if the iterate lies on a low-dimensional face of $B(F)$.
(The size of $\mathcal B^{(i)}$ may depend on the solver used for \eqref{eq:dual-subproblem};
to reduce the size of $\mathcal B^{(i)}$, we can minimize a random linear
objective over the optimal set of \eqref{eq:dual-subproblem} as in
\cite{udell2016bounding}.)

\paragraph{Linear convergence.}
Lacoste-Julien and Jaggi \cite{Lacoste2015} show that {\sc FCFW}
converges linearly to an $\epsilon$-suboptimal solution
when $g$ is smooth and strongly convex
so long as the active set $\mathcal B^{(i)}$ and iterate $x^{(i)}$
satisfy three conditions they call \emph{approximate correction}($\epsilon$):
\begin{enumerate}
\item \textbf{Better than {\sc FW}.} \(h(y^{(i)}) \leq \min_{\lambda \mathop{\in} [0, 1]}h((1-\lambda)w^{(i-1)} + \lambda v^{(i-1)}))\).
\item \textbf{Small away-step gap.} \(\max\{(w^{(i)} - v)^{\top} x^{(i)}: v \mathop{\in} \mathcal{V}(w^{(i)})\} \leq \epsilon\), where $\mathcal{V}(w^{(i)}) = \{v \in \mathcal V^{(i-1)}: \lambda_v > 0\}$.
\item \textbf{Representation.} $x^{(i)} \in \conv(\mathcal B^{(i)})$.
\end{enumerate}
By construction, iterates of {\sc L-FCFW} always satisfy these conditions with $\epsilon=0$:
\begin{enumerate}
\item \textbf{Better than FW.} For any $\lambda \in [0,1]$, $w = (1-\lambda)w^{(i-1)} + \lambda v^{(i-1)}$ is feasible.
\item \textbf{Zero away-step gap.} For each $v \mathop{\in} \mathcal{V}^{(i)}$,
if $w^{(i)} = v$, then clearly $(w^{(i)} - v)^{\top}(x^{(i)}) = 0$.
otherwise (if $w^{(i)} \ne v$)
$v - w^{(i)}$ is a feasible direction,
and so by optimality of $w^{(i)}$
$(w^{(i)} - v)^{\top}(x^{(i)}) \leq 0$.
\item \textbf{Representation.} We have $w^{(i)} \in \conv(\mathcal B^{(i)})$ by construction of $\mathcal B^{(i)}$.
\end{enumerate}
Hence we have proved \cref{thm:l-fcfw}:
\begin{theorem}
\label{thm:l-fcfw}
Suppose \(g\) is \(\alpha\)-smooth and \(\beta\)-strongly convex. Let \(M\) be the diameter of \(B(F)\) and \(\delta\) be the pyramidal width\footnote{See Appendix \ref{app:p_width} for definitions of the diameter and pyramidal width.} of \(P\), then the lower bounds \(d^{(i)}\) in {\sc L-FCFW} (\cref{alg:fcfw}) converges linearly at the rate of \(1-\rho\), i.e. \(p^{\star} - d^{(i+1)} \leq(1-\rho)(p^{\star} - d^{(i)})\), where \(\rho \overset{\Delta}{=}\frac{\beta}{4\alpha}(\frac{\delta}{M})^2\).
\label{convergence1}
\end{theorem}

\paragraph{Primal-from-dual algorithm.}
Recall that dual iterates yield primal iterates via \cref{thm:primal-from-dual}.
Hence the gradients $x^{(i)} = -\nabla g^*(-w^{(i)})$ computed by
{\sc L-FCFW} converge linearly to the solution $x^\star$ of \eqref{eq:primal}.
However, it is difficult to run {\sc L-FCFW} directly to solve \eqref{eq:dual}
given only access to $g$,
since in that case computing $g^*$ and its gradient requires solving another optimization problem;
moreover, we will see below that {\sc L-KM} computes the same iterates.
See \cref{s:primal-from-dual} for more discussion.

\journal{Consider approximate subproblem solves. 
Show how to ensure conditions for ``approximate correction'', possibly by taking a few away steps (?)
}

\section{L-KM (and OSM) converge linearly}\label{sec:convergence}

{\sc L-KM} (\cref{alg:L-KM}) and {\sc L-FCFW} (\cref{alg:fcfw}) are dual algorithms in the following strong sense:
\begin{theorem}
\label{thm:parallelism}
Suppose \(g\) is \(\alpha\)-smooth and \(\beta\)-strongly convex. In {\sc L-FCFW} (\cref{alg:fcfw}), suppose we choose $\mathcal B^{(i)} = \mathcal A^{(i)} = \{v \in \mathcal V^{(i-1)}: v^\top x^{(i)} = {w^{(i)}}^\top x^{(i)} \}$.
Then
\ben
\item The primal iterates $x^{(i)}$ of {\sc L-KM} and {\sc L-FCFW} match.
\item The sets \(\mathcal{V}^{(i)}\) used at each iteration of {\sc L-KM} and {\sc L-FCFW} match.
\item The upper and lower bounds $p^{(i)}$ and $d^{(i)}$ of {\sc L-KM} and {\sc L-FCFW} match.
\een
\end{theorem}
\begin{corollary}
The active sets of {\sc L-FCFW} can be chosen to satisfy $|\mathcal{B}^{(i)}| \leq n+1$.
\end{corollary}

\begin{theorem}
\label{thm:convergence_lkm}
Suppose \(g\) is \(\alpha\)-strongly convex and let \(M\) be the diameter of \(B(F)\), the duality gap \(p^{(i)}-d^{(i)}\) in {\sc L-KM} (\cref{alg:L-KM}) converges linearly: \(p^{(i)}-d^{(i)} \leq (p^{\star} - d^{(i)}) + M^2/(2\beta)\) when \((p^{\star} - d^{(i)}) \geq M^2/(2\beta)\) and \(p^{(i)}-d^{(i)} \leq M\sqrt{2(p^{\star} - d^{(i)})/\beta}\) otherwise. Note that \(p^{\star} - d^{(i)}\) converges linearly by \cref{thm:l-fcfw}.
\end{theorem}

When \(g\) is smooth and strongly convex, {\sc OSM} and vanilla {\sc FCFW} are dual algorithms in the same sense when we choose \(\mathcal{B}^{(i)} = \mathcal{V}^{(i-1)}\). For details of the duality between {\sc OSM} and {\sc L-FCFW} see \cref{app:duality_osm}. Hence we have a similar convergence result for {\sc OSM}:
\begin{theorem}
\label{thm:convergence_osm}
Suppose \(g\) is \(\alpha\)-strongly convex and let \(M\) be the diameter of \(B(F)\), the duality gap \(p^{(i)}-d^{(i)}\) in {\sc OSM} converges linearly: \(p^{(i)}-d^{(i)} \leq (p^{\star} - d^{(i)}) + M^2/(2\beta)\) when \((p^{\star} - d^{(i)}) \geq M^2/(2\beta)\) and \(p^{(i)}-d^{(i)} \leq M\sqrt{2(p^{\star} - d^{(i)})/\beta}\) otherwise.
\end{theorem}

\begin{remark}
Note that \(p^{\star} - d^{(i)}\) converges linearly by \cref{thm:l-fcfw}, \cref{thm:convergence_lkm} and \cref{thm:convergence_lkm} imply {\sc L-KM} and {\sc OSM} converge linearly when \(g\) is smooth and strongly convex.
\end{remark}

Moreover, this connection generates a new way to prune the active set of {\sc L-KM} even further
using a primal dual solver:
we may use any active set $\mathcal B^{(i)} \supseteq \{w \in \mathcal V^{(i-1)}: \lambda_w >0\}$,
where $\lambda \in \reals^{|\mathcal V^{(i-1)}|}$ is a dual optimal solution to \eqref{eq:primal-subproblem}.
When strict complementary slackness fails, we can have $\mathcal B^{(i)} \subset \mathcal A^{(i)}$.

\if \oldproof

\section{{\sc L-FCFW}, Connections to {\sc FCFW} and Convergence Analysis}\label{sec:fcfw}
\textcolor{red}{We can call this section: Connections to the Dual and Convergence. Explain the fact that one can go back and forth between primal dual iterates. Leave the proofs as is for now.} 
In this section we present {\sc L-FCFW}, the dual of  {\sc L-KM}, as a special case of the Fully-Corrective Frank-Wolfe algorithm that uses limited memory. This will help us also provide a direct proof of finite convergence of {\sc L-KM}. The results in this section will be applicable to general polytopes, not just the submodular base polytopes. We assume that \(g\) is closed, which implies \({g^{*}}^{*} = g\). 


The dual of {\sc OSM} can be seen as a first-order method which exploits the fact that it is easy to do linear optimization over the base polytope \(B(F)\) \cite{Bach2013}. We use this duality to re-write {L-KM} in the dual form and derive the dual method from scratch for {\sc L-KM} in Appendix \ref{app:dual_method}. {\sc L-FCFW} is thus limited memory like {\sc L-KM}. 
For ease of readability and to better demonstrate the connections between {\sc L-FCFW}  and {\sc FCFW} method over general polytopes, we re-write {\sc L-FCFW} for the minimization of the convex function \(h(w) = g^{*}(-w)\) over the base polytope \(B(F)\) (see \cref{alg:fcfw}). For general polytopes $P$, $B(F)$ can be replaced verbatim with $P$ in Algorithm \cref{alg:fcfw}. 

\begin{algorithm}[!t]\footnotesize
\caption{{\sc L-FCFW}: Dual of Limited memory Kelley's Method}
\begin{algorithmic}[1]
\State \textbf{Input:} a convex function \(h: \mathbb{R}^n \rightarrow \mathbb{R}\), polytope \(B(F) \subseteq \mathbb{R}^{n}\), suboptimality parameter \(\epsilon > 0\)
\State \textbf{Output:} an approximate solution \(w^{\sharp} \in B(F)\) s.t. \(h(w^{\sharp}) - \epsilon \leq \min_{w \mathop{\in} B(F)} h(w)\)
\State \textbf{Initialization:} let \(w^{(0)}\in \text{vert}(B(F))\),\ \(x^{(0)} \in -\partial h(w^{(0)})\), \(v^{(0)} \in \arg \max_{v \mathop{\in} \text{vert}(B(F))} v^{\top}x^{(0)}\), \(\mathcal{V}^{(0)} := \{w^{(0)}\},\ v^{(0)}\), \({\bar{\Delta}}^{(0)} = (v^{(0)} - w^{(0)})^{\top}x^{(0)}\),\ \(i = 1\)
\While {\({\bar{\Delta}}^{(i-1)} > \epsilon\)}
 \State let \(w^{(i)} \in \arg\min_{w \mathop{\in} \mathop{\mathrm{conv}}(\mathcal{V}^{(i-1)})} h(w)\) \label{dual_subproblem}
 \State compute \(x^{(i)}\) \(\in -\partial h(w^{(i)})\)\label{wtox}
 \State compute \(v^{(i)} \in \arg \max_{v \mathop{\in} \text{vert}(B(F))} v^{\top}x^{(i)}\)
  \State let \(\mathcal{A}^{(i)} := \{v \in \mathcal{V}^{(i-1)}: v^{\top}x^{(i)} = {w^{(i)}}^{\top}x^{(i)}\}\)    \Comment set of tight vertices w.r.t subgradient $x^{(i)}$\label{active2}
 \State let \(\mathcal{V}^{(i)} := \mathcal{A}^{(i)}\) \(\cup\) \(\{v^{(i)}\}\)    \Comment update \(\mathcal{V}^{(i)}\)\label{active3}
 \State \(i = i + 1\)
\EndWhile
\State \Return \(w^{(i-1)}\)
\end{algorithmic}
\label{LMdual}
\end{algorithm}
\begin{algorithm}[!t]\footnotesize
\caption{{\sc FCFW}: Fully-Corrective Frank-Wolfe Method with approximate correction \textcolor{red}{(do we need fcfw now explicitly?)}\cite{Lacoste2015}}
\begin{algorithmic}[1]
\State \textbf{Input:} convex function \(h: \mathbb{R}^n \rightarrow \mathbb{R}\), polytope $P \subseteq \mathbb{R}^n$, suboptimality parameter \(\epsilon > 0\)
\State \textbf{Output:} an approximate solution \(w^{\sharp} \in P\) s.t. \(h(w^{\sharp}) - \epsilon \leq \min_{w \mathop{\in} P} h(w)\) 
\State \textbf{Initialization:} let \(w^{(0)}\in \text{vert}(P)\),\ \(x^{(0)} \in -\partial h(w^{(0)})\), \(v^{(0)} \in \arg \max_{v \mathop{\in} \text{vert}(P)} v^{\top}x^{(0)}\), \(\mathcal{V}^{(0)} := \{w^{(0)}\},\ v^{(0)}\), \({\bar{\Delta}}^{(0)} = (v^{(0)} - w^{(0)})^{\top}x^{(0)}\),\ \(i = 1\)
\While {\({\bar{\Delta}}^{(i-1)} > \epsilon\)}
 \State compute \(w^{(i)}\), \(\mathcal{V}^{(i)}\) such that  \label{fcfwconditions}
   \Statex \quad\quad (a). \(w^{(i)} \in\) conv(\(\mathcal{V}^{(i)}\))
   \Statex \quad\quad (b). \(h(w^{(i)}) \leq \min_{\lambda \mathop{\in} [0, 1]}h(w^{(i-1)} + \lambda(v^{(i-1)} - w^{(i-1)}))\) \Comment make some progress with FW step
   \Statex \quad\quad (c). \(\max_{v \mathop{\in} \mathcal{V}^{(i)}} (w^{(i)} - v)^{\top}(x^{(i)}) \leq \epsilon\)    \Comment make the away step gap small enough 
 \State compute \(v^{(i)} \in \arg \max_{v \mathop{\in} \text{vert}(P)} v^{\top}x^{(i)}\)
\State let \(\bar{\Delta}^{(i)} := (v^{(i)} - w^{(i)})^{\top}x^{(i)}\)   \Comment update the optimality gap
 \State \(i = i + 1\)
\EndWhile
\State \Return \(w^{(i-1)}\)
\end{algorithmic}
\label{fcfw}
\end{algorithm}

In the \(i{\mathrm{th}}\) iteration, {\sc L-FCFW} minimizes $h(\cdot)$ over \(\mathop{\mathrm{conv}}(\mathcal{V}^{(i-1)})\) which is a subset of the polytope $B(F)$ (or a general polytope $P$), to get an approximate solution \(w^{(i)}\). When the optimality gap is not desirable, we generate a new vertex $v^{(i)}$ by maximizing the subgradient $x^{(i)}$ of $-h(w^{(i)})$, and define $\mathcal{V}^{(i)}$ by expanding conv(\(\mathcal{A}^{(i)}\)) (the tight set of vertices with respect to $x^{(i)}$) to include $v^{(i)}$.  
Through duality, we have \(\bar{d}^{\star} = -p^{\star}\) and \(h(w^{(i)}) = -d^{(i)}\), i.e. the lower bounds in {\sc L-KM}.

To analyze the rate of convergence for {\sc L-KM}, we present the Fully-Corrective Frank-Wolfe ({\sc FCFW}) Method with approximate correction \cite{Lacoste2015} in \cref{fcfw} and show that {\sc L-FCFW} is a limited memory special case of {\sc FCFW} in \cref{caseoffcfw}.

\begin{theorem}\label{caseoffcfw}
The dual method {\sc L-FCFW} is a special case of {\sc FCFW} with approximate correction as the approximate solution \(w^{(i)}\) and the subset of vertices \(\mathcal{V}^{(i)}\) constructed in {\sc L-FCFW} satisfy the conditions in Line \ref{fcfwconditions}(a)-(c) of \cref{fcfw} in each iteration $i\geq 1$. 
\end{theorem}
\textbf{Proof sketch:} Since {\sc L-FCFW} computes $w^{(i)} \in \arg \min_{v \in \mathop{\mathrm{conv}}(\V^{(i-1)})} h(w)$, by first-order optimality conditions we get $w^{(i)\top}x^{(i)} \geq w^\top x^{(i)}$ for all $w \in \mathop{\mathrm{conv}}(\V^{(i-1)})$ (recall $x^{(i)} \in -\partial h(w^{(i)})$). Since $v \in \A^{(i)}$ for all $v \in \V^{(i-1)}$ such that \(v^\top x^{(i)} = w^{(i)\top}x^{(i)}\), only vertices in $\A^{(i)}$ can have non-zero convex multipliers in the decomposition of $w^{(i)}$ (see {\cref{winA}}, Appendix \ref{app:dual_method}). Therefore, $w^{(i)}\in \mathop{\mathrm{conv}}(\mathcal{A}^{(i)}) \subseteq \text{conv}(\mathcal{V}^{(i)})$ and (a) holds. Condition (b) holds because \(\mathop{\mathrm{conv}}(\{w^{(i-1)},\ v^{(i-1)}\}) \subseteq \mathop{\mathrm{conv}}(\mathcal{V}^{(i-1)})\), over which \(w^{(i)}\) is optimal. Lastly (c) holds because by construction of $\A^{(i)}$ and \(v^{(i)}\), \({({w^{(i)}} - v)}^{\top}x^{(i)}=0\) for \(v \in A^{(i)}\) and \({({w^{(i)}} - v)}^{\top}x^{(i)}<0\) when \(v =  v^{(i)}\).

Therefore, known convergence rates for {\sc FCFW} \cite{Lacoste2015} are applicable to {\sc L-FCFW}: 

\begin{theorem} 
Suppose \(g\) is closed, \(\frac{1}{\mu}\)-strongly smooth and \(\frac{1}{L}\)-strongly convex. Let \(M\) be the diameter and \(\delta\) be the pyramidal width\footnote{See Appendix \ref{app:dual_method} for definitions of the diameter and pyramidal width.} of \(P\), then the lower bounds \(d^{(i)}\) in \cref{alg:L-KM} converges linearly at the rate of \(1-\rho\), i.e. $p^{\star} - d^{(i+1)} \leq(1-\rho)(p^{\star} - d^{(i)})$, where \(\rho \overset{\Delta}{=}\frac{\mu}{4L}(\frac{\delta}{M})^2\).
\label{convergence1}
\end{theorem}

\begin{theorem}
Suppose \(g\) is closed, \(\frac{1}{L}\)-strongly convex and let \(M\) be the diameter of \(P\), the duality gap \(\Delta^{(i)}\) in Algorithm \ref{alg:L-KM} converges sub-linearly: $\Delta^{(i+1)} \leq (p^{\star} - d^{(i)}) + LM^2/2$ when $(p^{\star} - d^{(i)}) \geq LM^2/2$ and $\Delta^{(i)} \leq M\sqrt{2(p^{\star} - d^{(i)})L}$ otherwise.
\end{theorem}

\textcolor{red}{(change the theorems. revisit the following statement:)}
Note that the same analysis would not apply to the dual of {\sc OSM} since {\sc OSM} does not support deletion of constraints. Therefore, it is not possible to bound the away step gap in Line \ref{fcfwconditions}(c). 

\fi

\section{Experiments and Conclusion}\label{sec:experiments}
We present in this section a computational study: 
we minimize non-separable composite functions $g+f$ where \(g(x) = x^{\top}(A + n \mathbf{I}_n) x + b^{\top}x\) for \(x \in \mathbb{R}^{n}\), and $f$ is the Lov\'{a}sz extension of the submodular function \(F(A) \mathop{=} \frac{|A|(2n - |A| + 1)}{2}\) for $A \subseteq [n]$. To construct $g(\cdot)$, entries of \(A \in M_{n}\) and \(b \in \mathbb{R}^{n}\) were randomly sampled from $U[-1, 1]^{n \times n}$, and $U[0, n]^n$ respectively.
 $\mathbf{I}_n$ is an $n\times n$ identity matrix.
\begin{figure}[!t]
\centering
 \begin{subfigure}{3.3cm}
  \centering\includegraphics[width=3.3cm]{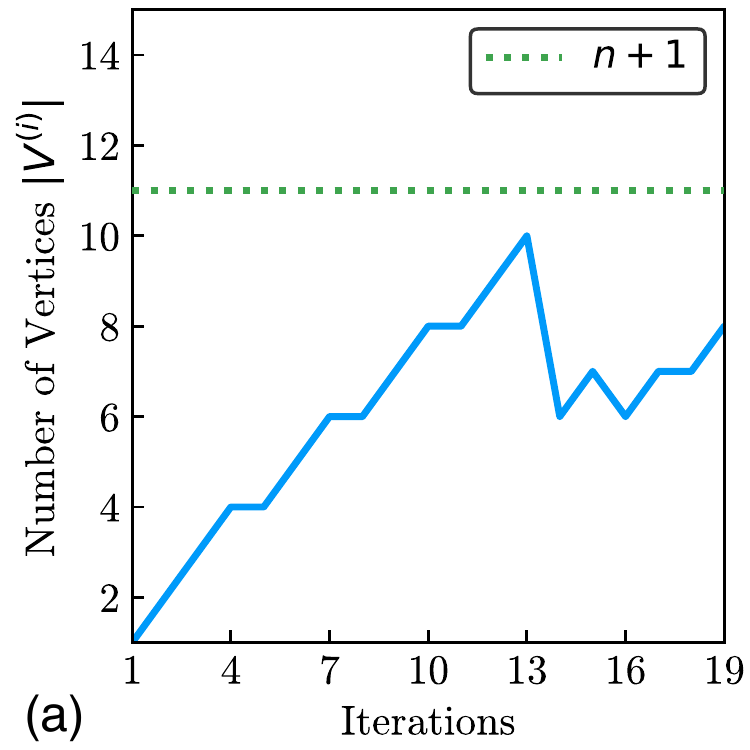}
 \end{subfigure}
 \begin{subfigure}{3.3cm}
\centering\includegraphics[width=3.3cm]{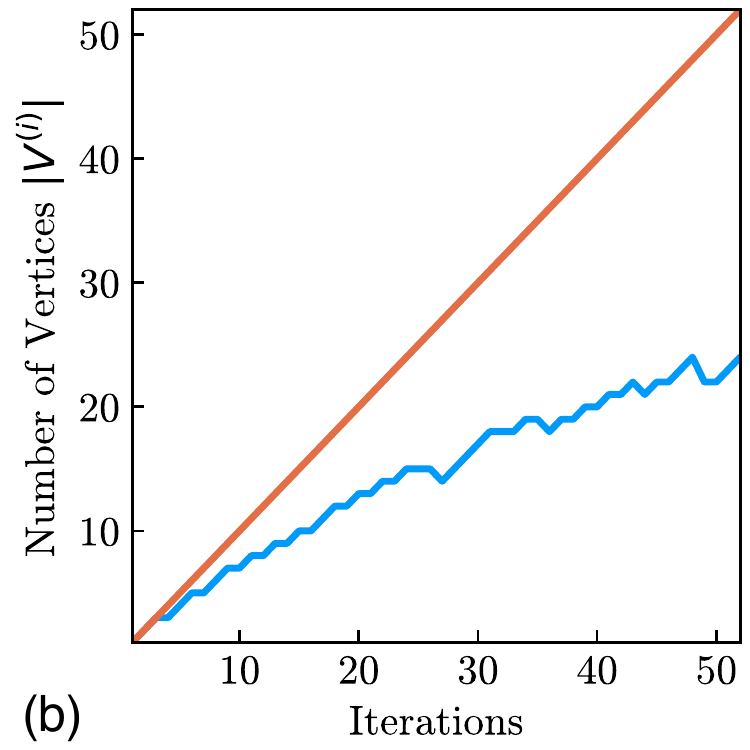}
 \end{subfigure}
 \begin{subfigure}{3.3cm}
 \centering\includegraphics[width=3.3cm]{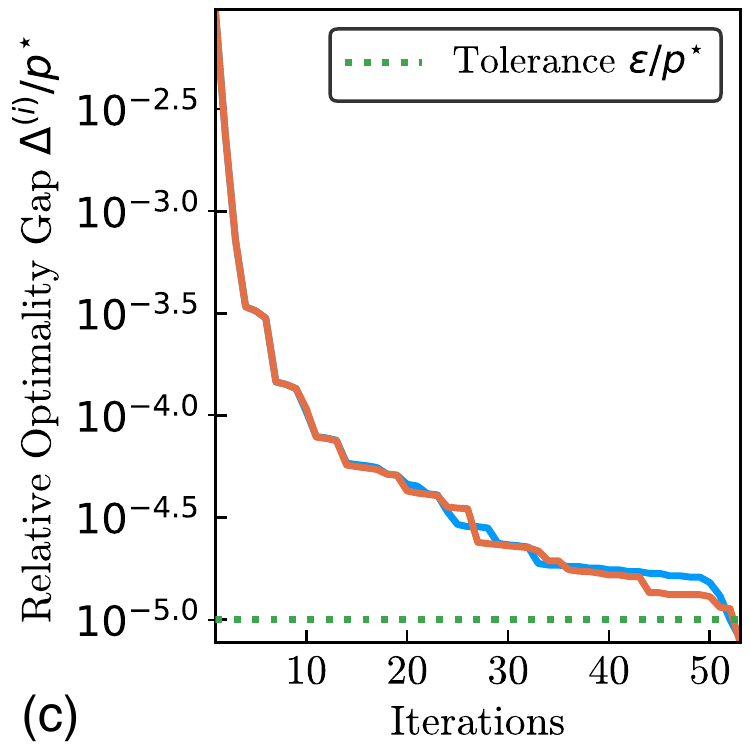}
 \end{subfigure}
 \begin{subfigure}{3.3cm}
 \centering\includegraphics[width=3.3cm]{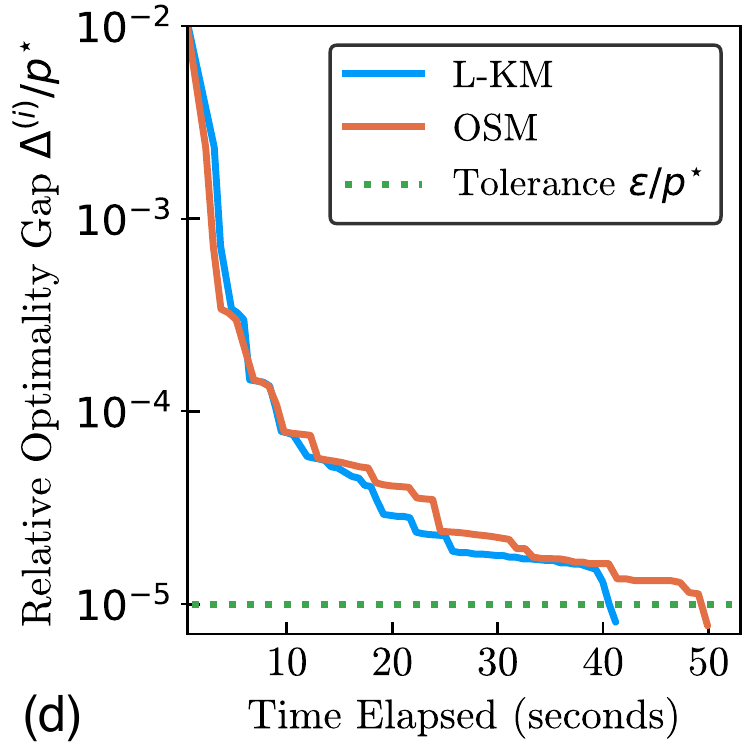}
 \end{subfigure}
 \caption{\footnotesize Dimension \(n = 10\) in (a), \(n = 100\) in (b), (c) and (d). The methods converged in (a), (b), (c) and (d).}
  \label{ori_and_ltd_fig}
\end{figure}
We remark that {\sc L-KM} converges so quickly that the bound on the size of the active set is less important, in practice,
than the fact that the active set need not grow at every iteration.

\textbf{Primal convergence}: We first solve a toy problem \textbf{Primal convergence}: We first solve a toy problem for \(n = 10\) and show that the number of constraints does not exceed \(n + 1\). Note that the number of constraints might oscillate before it reaches \(n + 1\) (\cref{ori_and_ltd_fig}(a)). We next compare the memory used in each iteration (\cref{ori_and_ltd_fig}(b)), the optimality gap per iteration (\cref{ori_and_ltd_fig}(c)), and the running time (\cref{ori_and_ltd_fig}(d)) of {\sc L-KM} and {\sc OSM} by solving the problem for \(n = 100\) up to accuracy of \(10^{-5}\) of the optimal value. Note that {\sc L-KM} uses much less memory compared to {\sc OSM}, converges at almost the same rate in iterations, and its running time per iteration improves as the iteration count increases.
\begin{figure}[!t]
\centering
 \begin{subfigure}{3.3cm}
 \centering\includegraphics[width=3.3cm]{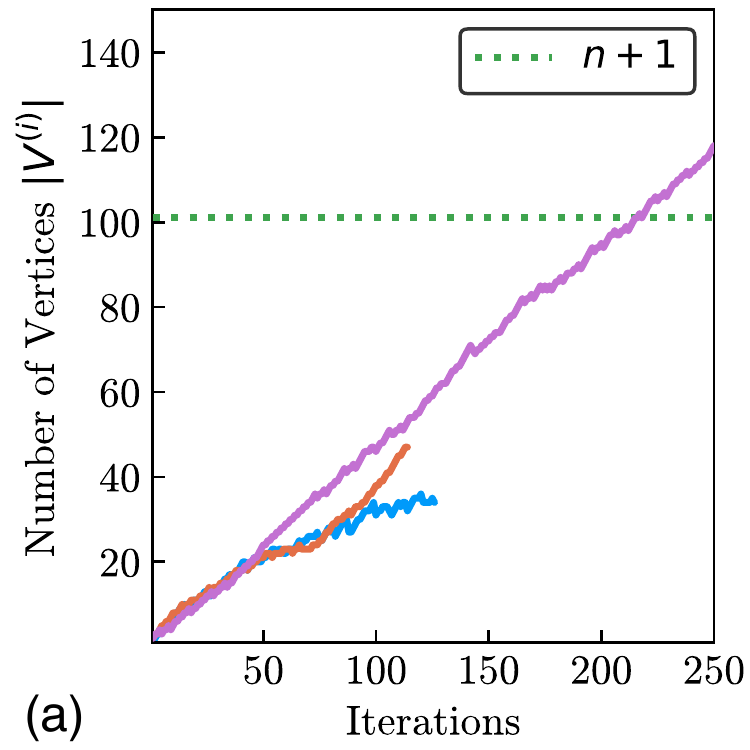}
 \end{subfigure}
 \begin{subfigure}{3.3cm}
 \centering\includegraphics[width=3.3cm]{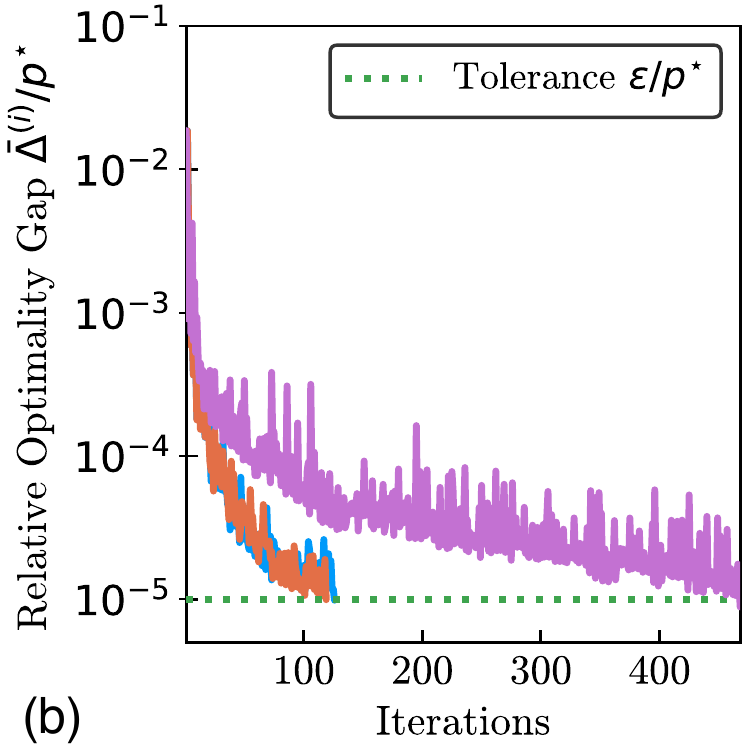}
 \end{subfigure}
 \begin{subfigure}{3.3cm}
 \centering\includegraphics[width=3.3cm]{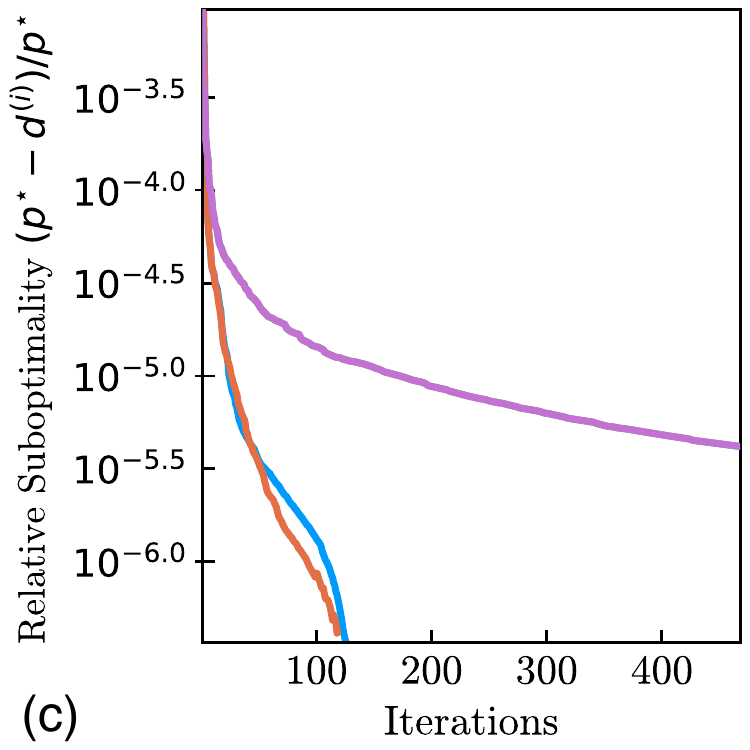}
 \end{subfigure}
 \begin{subfigure}{3.3cm}
 \centering\includegraphics[width=3.3cm]{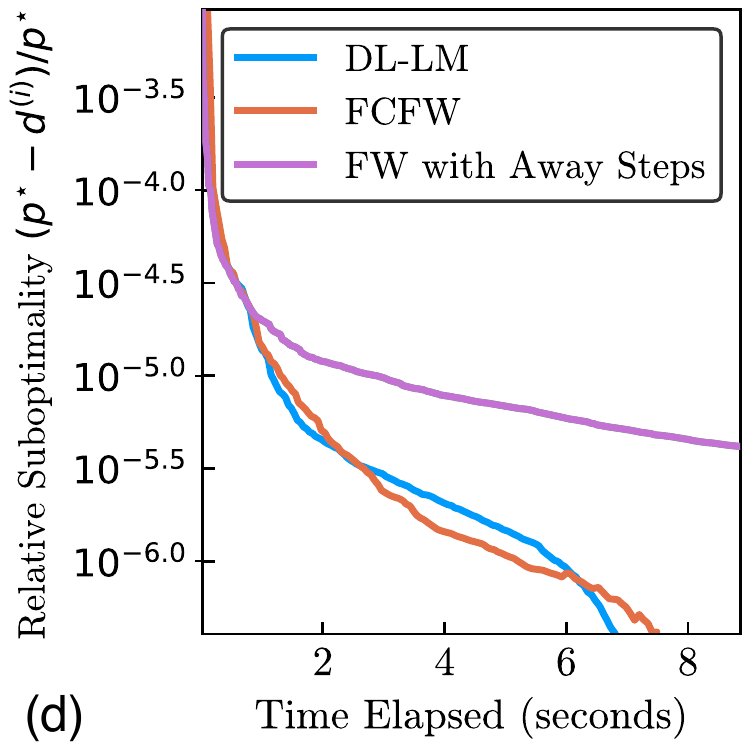}
 \end{subfigure}
 \caption{\footnotesize {\sc L-FCFW} and {\sc FCFW} converged in all plots, {\sc FW} with away steps has converged in (b), (c) and (d).}
 \label{dual_fw_fig}
\end{figure}

\textbf{Dual convergence}: We compare the convergence of {\sc L-FCFW}, {\sc FCFW} and Frank-Wolfe with away steps for the dual problem \(\max_{w \mathop{\in} B(F)}-(-w - b)^{\top}(A + n \mathbf{I}_n)^{-1}(-w-b)\) for \(n = 100\) up to relative accuracy of \(10^{-5}\). {\sc L-FCFW} maintains smaller sized subproblems (Figure \eqref{dual_fw_fig}(a)), and it converges faster than {\sc FCFW} as the number of iteration increases (Figure \eqref{dual_fw_fig}(d)). Their provable duality gap converges linearly in the number of iterations. Moreover, as shown in Figures \eqref{dual_fw_fig}(b) and (c), {\sc L-FCFW} and {\sc FCFW} return better approximate solutions than Frank-Wolfe with away steps under the same optimality gap tolerance.

\textbf{Conclusion} This paper defines a new limited memory version of Kelley's method adapted to composite convex and submodular objectives, and establishes the first convergence rate for such a method, solving the open problem proposed in \cite{Bach2015,Bach2013}.
We show bounds on the memory requirements and convergence rate, and demonstrate compelling performance in practice.

\subsubsection*{Acknowledgments}
This work was supported in part by DARPA Award FA8750-17-2-0101. A part of this work was done while the first author was at the Department of Mathematical Sciences, Tsinghua University and while the second author was visiting the Simons Institute, UC Berkeley. The authors would also like to thank Sebastian Pokutta for invaluable discussions on the Frank-Wolfe algorithm and its variants.







\footnotesize
\bibliographystyle{plain}
\bibliography{related_m}
\medskip
\small

\newpage
\appendix
\section{Additional Background} \label{app:background}
\subsection{Additional Examples}
We list some examples of popular submodular functions in Table \ref{tab:listofproblems}.

\begin{table}[ht]\footnotesize
\centering
\begin{tabular}{|p{6.0cm}|p{6.0cm}|}\hline
\textbf{Problem}	 & \textbf{Submodular function, $S \subseteq E$ (unless specified)}\\ \hline\hline
n experts (simplex), $E = \{1, \hdots, n\}$  & $f(S) = 1$\\ \hline
k out of n experts (k-simplex), $E = \{1, \hdots, n\}$ & $f(S) = \min\{|S|, k\}$ \\ \hline
Permutations over $E = \{1, \hdots, n\}$ & $f(S) = \sum_{s=1}^{|S|} (n+1-s) $  \\\hline 
k-truncated permutations over $E = \{1, \hdots, n\}$ & $f(S) = (n-k)|S|$ for $|S| \leq k$, $f(S) = k(n-k)  + \sum_{s=k+1}^{|S|} (n+1-s)$ if $|S| \geq k$  \\\hline
Spanning trees on $G=(V,E)$  & $f(S) = |V(S)| - \kappa(S)$, $\kappa(S)$ is the number of connected components of $S$ \\\hline
Matroids over ground set $E$: $M = (E, \mathcal{(I)}), \mathcal{(I)} \subseteq 2^{E}$ & $f(S) = r_M(S)$, the rank function of the matroid \\\hline
Coverage of T: given $T_1, \hdots, T_n \subseteq T$ &  $f(S) = |\bigcup_{i\in S} T_i|$, $E = \{1, \hdots, n\}$  \\\hline
Cut functions on a directed graph $D = (V, E)$, $c: E \rightarrow \mathbb{R}_+$ & $f(S) = c(\delta^{out}(S))$, $S \subseteq V$\\\hline
Flows into a sink vertex $t$, given a directed graph $D= (V, E)$ and costs $c: E \rightarrow \mathbb{R}_+$ & $f(S) = $ max flow from $S \subseteq V \setminus \{t\}$ into $t$ \\\hline
Maximal elements in $E$, $h: E \rightarrow \mathbb{R}$  & $f(S) = \max_{e \in S} h(e)$, $f(\emptyset) = \min_{e \in E} h(e)$  \\\hline 
Entropy $H$ of random variables $X_1, \hdots, X_n$ & $f(S) = H(\bigcup_{i \in S} X_i)$, $E = \{1, \hdots, n\}$\\\hline
\end{tabular}
\caption{Problems and the submodular functions (on ground set of elements $E$) that give rise to them.}
\label{tab:listofproblems} 
\end{table}

\subsection{Strong Convexity and Smoothness}\label{app:convexity}
We say a function \(g: \mathbb{R}^n \rightarrow \mathbb{R}\) is \(\alpha\)-strongly convex if \(g(x) - \alpha/2 \|x\|^2\) is convex, where \(\alpha > 0\). It is easy to see that the sum of a stronly convex function and a piecewise linear function is still strongly convex, and we have
\begin{lemma}
\label{lemma:uniqueness}
When \(g: \mathbb{R}^n \rightarrow \mathbb{R}\) is a strongly convex function, then
\begin{equation}
\tag{\textit{P(V)}} 
\bopt
\minimize & g(x) + \max_{w \mathop{\in} \mathop{\mathrm{conv}(\mathcal{V})}}w^{\top}x \\
\eopt
\end{equation}
has a unique optimal solution \(x^{\star}\) for all \(\mathcal{V} \subseteq \mathbb{R}^n\).
\end{lemma}

On the other hand, we say a function \(g: \mathbb{R}^n \rightarrow \mathbb{R}\) is \(\beta\)-smooth if there exists \(\beta > 0\) such that \(g(x) - \beta/2\|x\|^2\) is concave. We have\cite{Ryu2016}:
\begin{lemma}
\label{lemma:convexity_smoothness}
When a function \(g\) is \(\alpha\)-strongly convex, its Fenchel conjugate \(g^*\) is \(\frac{1}{\alpha}\)- smooth.
\end{lemma}

\begin{lemma}
\label{lemma:smoothness_convexity}
When a function \(g\) is \(\beta\)-smooth, its Fenchel conjugate \(g^*\) is \(\frac{1}{\beta}\)-strongly convex.
\end{lemma}

\section{The Original Simplicial Method (Section \ref{sec:lkm})}
\label{app:osm}
We present the Original Simplicial Method ({\sc OSM}) in Algorithm \ref{alg:osm}.

\begin{algorithm}[ht]\footnotesize
\caption{{\sc OSM}: The Original Simplicial Method for \eqref{eq:primal}
\label{alg:osm}}
\begin{algorithmic}[1]
\Require strongly convex function \(g: \mathbb{R}^n \rightarrow \mathbb{R}\), 
submodular function \(F: 2^n \rightarrow \mathbb{R}\), 
tolerance \(\epsilon > 0\)
\Ensure $\epsilon$-suboptimal solution \(x^{\sharp}\)
\State initialize:
choose \(x^{(0)} \in \mathbb{R}^{n}\), 
set \(\mathcal{V}^{(0)} = \emptyset\)
\For {i=1,2,\ldots}
 \State \textbf{Convex subproblem.} Define approximation \(f_{(i)}(x) = \max\{w^{\top}x : w \mathop{\in} \mathcal{V}^{(i-1)}\}\) and solve
 \[
 x^{(i)} = \argmin g(x) + f_{(i)}(x).
 \]
 \State \textbf{Submodular subproblem.} Compute value and subgradient of $f$ at $x^{(i)}$ 
 \[
 f(x^{(i)}) = \max_{w \in B(F)} w^{\top} x^{(i)}, \qquad v^{(i)} \mathop{\in} \partial f(x^{(i)}) = \argmax_{w \in B(F)} w^{\top} x^{(i)}.
 \]
\State \textbf{Stopping condition.} Break if duality gap \(p^{(i)} - d^{(i)} \leq \epsilon\), where
    \[
    p^{(i)} = g(x^{(i)}) + f(x^{(i)}), \qquad d^{(i)} = g(x^{(i)}) + f_{(i)}(x^{(i)}).
    \]
  \State \textbf{Update memory.} Update memory \(\mathcal{V}^{(i)}\):
 \[
 \mathcal{V}^{(i)} = \mathcal{V}^{(i-1)} \cup \{v^{(i)}\}.
 \]
\EndFor
\State \Return \(x^{(i)}\)
\end{algorithmic}
\end{algorithm}

\section{Limited Memory Kelley's Method (Section \ref{sec:lkm})}
\label{app:lkm}
In this section, we provide proofs of some of the results in Section \ref{sec:lkm}.

\textbf{Proof of \cref{thm:limited}.}
\begin{proof}
We prove this by induction. The claim is true for \(i = 0\) since \(\mathcal{V}^{(0)}\) has only one element. Suppose that the claim is true for \(i < i_0\). When \(\Delta > \epsilon\), we have \({v^{(i_0)}}^{\top}x^{(i_0)} = f(x^{(i_0)}) > f_{(i_0)}(x^{(i_0)})\). From \(\mathcal{A}^{(i_0)} \subseteq \{w \in \mathbb{R}^n\mid w^{\top}x^{(i_0)} = f_{(i_0)}(x^{(i_0)})\}\) we have \(v^{(i_0)} \notin \mathop{\mathrm{affine}}(\mathcal{A}^{(i_0)})\). Otherwise when \(\Delta^{(i)} \leq \epsilon\), the algorithm terminates in the \(i_{0}\)th iteration.

Since vectors in \(\mathcal{V}^{(i)}\) are affinely independent, we have \(|\mathcal{V}^{(i)}| \leq n+1\) for all \(i\) since \(\mathcal{V}^{(i)} \subseteq \mathbb{R}^n\).
\end{proof}

Before proving \cref{lemma:equality}, we first present a lemma that is used in the proof of \cref{lemma:equality}:

\begin{lemma}
\label{lemma:eq_ball}
Given a submodular function \(F: 2^{V} \to \mathbb{R}\), let \(\mathcal{W} \subseteq \mathop{\mathrm{vert}}(B(F))\) be a subset of the vertices of its base polytope. For the piecewise linear function
\[
\tilde{f}(x) = \max_{w \mathop{\in} \mathop{\mathrm{conv}}(\mathcal{W})}w^{\top}x,
\]
let \(\mathcal{A}(x) \overset{\Delta}{=} \{w^{\star} \in \mathcal{W} \mid {w^{\star}}^{\top}x = \tilde{f}(x)\}\) be the points in \(\mathcal{W}\) that are active at \(x\). Then given any \(\bar{x} \in \mathbb{R}^n\), there exists \(\epsilon > 0\) such that 
\[
\tilde{f}(x) = \max_{w^{\star} \mathop{\in} \mathop{\mathrm{conv}}(\mathcal{A}(\bar{x}))}{w^{\star}}^{\top}x
\]
for all \(x \in \mathcal{B}(\bar{x},\ \epsilon)\).
\end{lemma}
\begin{proof}

Since \(\mathcal{W}\) is finite, we have \(\tilde{f}(\bar{x}) \geq \max_{\tilde{w} \in \mathcal{W} \setminus \mathcal{A}(\bar{x})}\tilde{w}^{\top}\bar{x} + \epsilon\), where \(\epsilon > 0\). Let \(L = \max_{w \in \mathcal{W}}\|w\|\), then for all \(x \in \mathcal{B}(\bar{x},\ \epsilon/(3L))\), \(w^{\star} \in \mathcal{A}(\bar{x})\) and \(\tilde{w} \in \mathcal{W} \setminus \mathcal{A}(\bar{x})\), we have
\begin{equation}
    \begin{split}
        {w^{\star}}^{\top}x - \tilde{w}^{\top}x &= ({w^{\star}}-\tilde{w})^{\top}\bar{x} + {w^{\star}}^{\top}(x - \bar{x}) + \tilde{w}^{\top}(\bar{x} - x) \\
        &\geq \epsilon - L\frac{\epsilon}{3L} - L\frac{\epsilon}{3L} \\
        &= \frac{\epsilon}{3}.
    \end{split}
\end{equation}
Hence \(\tilde{f}(x) > \tilde{w}^{\top}x\) for all \(x \in \mathcal{B}(\bar{x},\ \epsilon/(3L))\) and \(\tilde{w} \in \mathcal{W} \setminus \mathcal{A}(\bar{x})\), which is equivalent to \(\tilde{f}(x) = \max_{w^{\star} \mathop{\in} \mathop{\mathrm{conv}}(\mathcal{A}(\bar{x}))}{w^{\star}}^{\top}x\) for all \(x \in \mathcal{B}(\bar{x},\ \epsilon/(3L))\).
\end{proof}

\textbf{Proof of \cref{lemma:equality}.}
\begin{proof}
Let \(P_{(i)}(x)\overset{\Delta}{=} \min g(x) + \max_{w \mathop{\in} \mathop{\mathrm{conv}}(\mathcal{V}^{(i-1)})}w^{\top}x = g(x) + f_{(i)}(x)\) and \(\widetilde{P}_{(i)} \overset{\Delta}{=} \min_{x \mathop{\in} \mathbb{R}^n} g(x) + \max_{w \mathop{\in} \mathop{\mathrm{ conv}}(\mathcal{A}^{(i)})}w^{\top}x\). There exists at least one $w^{\star} \in \V^{(i-1)}$ such that $f_{(i)}(x^{(i)}) = w^{\star T}x^{(i)}$. Therefore, $P_{(i)}(x^{(i)}) = g(x^{(i)}) + w^{\star T}x^{(i)} = g(x^{(i)}) + \max_{w \in \mathop{\mathrm{conv}}(\mathcal{A}^{(i)})} w^{\top}x^{(i)} = \widetilde{P}_{(i)}(x^{(i)})$, where the last equality follows from the definition of $\A^{(i)}$. Next, if we can show local optimality of $x^{(i)}$ for $\widetilde{P}_{(i)}$, this would imply global optimality of $x^{(i)}$ for $\widetilde{P}_{(i)}$ due to convexity of $\widetilde{P}_{(i)}$, thus \({P}_{(i)}\) and $\widetilde{P}_{(i)}$ will have the same optimal value. By the definition of \(\mathcal{A}^{(i)}\) and \cref{lemma:eq_ball}, we have \(f(x) = \max_{w \in \mathop{\mathrm{conv}}(\mathcal{A}(x^{(i)}))} = f_{(i)}(x^{(i)})\) in \(\mathcal{B}(x^{(i)},\ \epsilon)\) for some \(\epsilon > 0\). Thus \(P_{(i)}(x^{(i)}) = g(x) + f(x) = g(x) + f_{(i)}(x) = \widetilde{P}_{(i)}(x)\) for \(x \in \mathcal{B}(x^{(i)},\ \epsilon)\). Hence \(x^{(i)}\) is an local optimal solution to \(\widetilde{P}_{(i)}\), and the lemma is proved. By \cref{lemma:uniqueness}, \(x^{(i)}\) is the unique solution to both \(P_{(i)}\) and \(\widetilde{P}_{(i)}\).

\end{proof}

\textbf{Proof of Corollary \ref{cor:lowerbounds}.}
\begin{proof}
For any \(i \geq 1\), by \cref{lemma:uniqueness}, there exists an \(x^{(i)} \in \mathbb{R}^n\) that minimizes \(g(x) + f_{(i)}(x)\). Thus we have
\begin{align}
    d^{(i)} &= g(x^{(i)}) + f_{(i)}(x^{(i)}) \nonumber\\
    & = g(x^{(i)}) + \max_{w \in \mathop{\mathrm{conv}}(\mathcal{V}^{(i-1)})}w^{\top}x^{(i)} \nonumber\\
    & \geq g(x^{(i)}) + \max_{w \in \mathop{\mathrm{conv}}(\mathcal{A}^{(i-1)})}w^{\top}x^{(i)}  &\triangleright\ \mathcal{A}^{(i-1)} \subseteq \mathcal{V}^{(i-1)} \\
    &>  g(x^{(i-1)}) + \max_{w \in \mathop{\mathrm{conv}}(\mathcal{A}^{(i-1)})}w^{\top}x^{(i-1)}  &\triangleright\ \text{optimality and uniqueness of }x^{(i-1)} \nonumber\\
    &= d^{(i-1)}. \nonumber
\end{align}
On the other hand, by \(\mathcal{V}^{(i-1)} \subseteq \mathop{\mathrm{vert}}(B(F))\), we have
\begin{equation}
    \begin{split}
        d^{(i)} &= \min_{x \mathop{\in} \mathbb{R}^n} \{g(x) + \max_{w \in \mathop{\mathrm{conv}}(\mathcal{V}^{(i-1)})}w^{\top}x\} \\
        &\leq \min_{x \mathop{\in} \mathbb{R}^n}\{ g(x) + \max_{w \in B(F)}w^{\top}x\} \\
        &= \min_{x \mathop{\in} \mathbb{R}^n} g(x) + f(x)
    \end{split}
\end{equation}
for all \(i \geq 0\).
\end{proof}

\textbf{Proof of Corollary \ref{cor:inequality}.}
\begin{proof}
Note that each \(\mathcal{V}^{(i)}\) determines a unique \(d^{(i)}\). Suppose for contradiction that there exists \(i_1 \neq i_2\) but \(\mathcal{V}^{(i_1)} = \mathcal{V}^{(i_2)}\), then we will have \(d^{(i_1)} = d^{(i_2)}\), which contradicts the fact that \{\(d^{(i)}\)\} strictly increases.
\end{proof}

\textbf{Proof of \cref{thm:termination}}

\begin{proof}
Since \(\mathop{\mathrm{vert}}(B(F))\) has finitely many vertices, there are only finitely many choices of \(\mathcal{V}^{(i)} \subseteq \mathop{\mathrm{vert}}(B(F))\). Thus by \cref{cor:inequality}, \cref{alg:L-KM} terminates within finitely many steps.

Suppose for contradiction that when the algorithm terminates at \(i = i_0\), \(p^{(i_0)} - d^{(i_0)} > \epsilon \geq 0\). Let \(\mathcal{A}^{(i_0)} \overset{\Delta}{=} \{w \in \mathcal{V}^{(i_0-1)}: w^{\top}{x^{(i_0)}} \overset{\Delta}{=} f_{(i_0)}(x^{(i_0)})\) and \(v^{(i_0)} \in \mathcal{V}(x^{(i_0)})\). Define \(\mathcal{V}^{(i_0)} \overset{\Delta}{=} \mathcal{A}^{(i_0)} \cup \{v^{(i_0)}\}\) and \(f_{(i_0+1)}(x) = \max\{w^{\top}x : w \mathop{\in} \mathcal{V}^{(i_0)}\}\), then let \(x^{(i_0+1)} = \argmin_{x \mathop{\in} \mathbb{R}^n} g(x) + f_{(i_0+1)}(x)\). By the proof of \cref{cor:lowerbounds}, we have \(d^{(i_0+1)} = g(x^{(i_0+1)}) + f_{(i_0+1)}(x^{(i_0+1)}) > d^{i_0}\), so \(\mathcal{V}^{(i_0)}\) is different to any \(V^{(i)}\) where \(i \leq i_0\), and {\sc L-KM} should not have terminated at \(i = i_0\). Thus {\sc L-KM} would never terminate when \(p^{(i)} - d^{(i)} > \epsilon \geq 0\).

\end{proof}

\section{Duality (Section \ref{s:duality})} \label{app:duality}

To prove the strong duality between \eqref{eq:primal-subproblem} and \eqref{eq:dual-subproblem}, we first verify the weak duality:

\begin{theorem}[Weak Duality]
The optimal value of primal problem \eqref{eq:primal-subproblem} is greater than or equal to the optimal value of the dual problem \eqref{eq:dual-subproblem}. \label{thm:weak-duality-subproblems}
\end{theorem}

\begin{proof}
We first have
\begin{equation}
\min_{x \mathop{\in} \mathbb{R}^n} \{g(x) + \max_{w \mathop{\in} \mathop{\mathrm{conv}}(\mathcal{V})}w^{\top}x\} = \min_{x \mathop{\in} \mathbb{R}^n} \max_{w \mathop{\in} \mathop{\mathrm{conv}}(\mathcal{V})} g(x) + w^{\top}x.
\label{ftow}
\end{equation}
For any given \(\tilde{w} \in \mathop{\mathrm{conv}}(\mathcal{V})\), we also have
\begin{equation}
\min_{x \mathop{\in} \mathbb{R}^n} \max_{w \mathop{\in} \mathop{\mathrm{conv}}(\mathcal{V})} g(x) + {\tilde{w}}^{\top}x \geq \min_{x \mathop{\in} \mathbb{R}^n} g(x) + {\tilde{w}}^{\top}x.
\end{equation}
Thus by the definition of \(g^{*}\), we can see that
\begin{equation}
\begin{split}
\min_{x \mathop{\in} \mathbb{R}^n} \max_{w \mathop{\in} \mathop{\mathrm{conv}}(\mathcal{V})} g(x) + w^{\top}x 
\mathop{\geq}& \max_{w \mathop{\in} \mathop{\mathrm{conv}}(\mathcal{V})} \min_{x \mathop{\in} \mathbb{R}^n} g(x) + w^{\top}x \\
\mathop{=}& \max_{w \mathop{\in} \mathop{\mathrm{conv}}(\mathcal{V})} - \max_{x \mathop{\in} \mathbb{R}^n} (-w)^{\top}x - g(x) \\
\mathop{=}& \max_{w \mathop{\in} \mathop{\mathrm{conv}}(\mathcal{V})} - g^{*}(-w).
\end{split}
\label{wtog}
\end{equation}
Combine {\eqref{ftow}} and {\eqref{wtog}}, and the theorem follows.
\end{proof}

\textbf{Proof of \cref{thm:strong-duality-subproblems}}

\begin{proof}
By \cref{lemma:uniqueness}, we know \eqref{eq:primal-subproblem} has a unique solution \(\bar{x}\). Since \(g\) is convex, we have \(\partial g(x) \neq \emptyset\). By the optimality of \(\bar{x}\), we also have \(0 \in \partial g(\bar{x}) + \partial f(\bar{x})\). Let \(\bar{w} \in -\partial g(\bar{x}) \cap \partial f(\bar{x})\), then
\begin{equation}
g^{*}(-\bar{w}) = (-\bar{w})^{\top}\bar{x} - g(\bar{x})
\end{equation}
by \cref{lemma:Fenchelpair}. Note that \(\bar{w} \in \partial f(\bar{x})\), we also have \(f(\bar{x}) = {\bar{w}}^{\top}\bar{x}\) by Equation \eqref{lemma:lovasz_attainment}. Thus
\begin{equation}
    f(\bar{x}) + g(\bar{x}) = g^*({\bar{w}}),
\end{equation}
\(\bar{w}\) is an optimal solution to \eqref{eq:dual-subproblem} and we have \eqref{eq:primal-subproblem} and \eqref{eq:dual-subproblem} via weak duality.
\end{proof}

\section{Primal-from-dual algorithm (Section \ref{s:solving-dual})} \label{s:primal-from-dual}
Now consider the \emph{Primal-from-dual} algorithm presented in \cref{s:solving-dual}.

Formally, assume $g$ is $\alpha$-strongly convex.
Suppose we obtain $w \in B(F)$ with
\[
\|w - w^{\star}\| \leq \epsilon
\]
via some dual algorithm (e.g., {\sc L-FCFW}). 
Define $x = \nabla_w (-g^*(-w)) = \argmin_x g(x) + w^\top x$.
Since $g^*$ is $1/\alpha$ smooth, we have
\[
\|x - x^\star\| \leq 1/\alpha \|w - w^\star\| \leq \epsilon/\alpha
\]
Hence if the dual iterates converge linearly, so do the primal iterates.

The remaining difficulty is how to solve the {\sc L-FCFW} subproblems. 
One possibility is to use the values and gradients of
(a {\sc first order oracle} for) $h = g^*$.
To implement a first order oracle for $h = g^*$, we need only 
solve an unconstrained minimization problemma:
\[
g^*(y) = \max_{x \mathop{\in}\mathbb{R}^n} y^\top x - g(x),
\qquad
\nabla g^*(y) = \argmax_{x \mathop{\in}\mathbb{R}^n} y^\top x - g(x).
\]
This problem is straightforward to solve 
since $g$ is smooth and strongly convex.
However, it is not clear how solving these subproblems approximately 
affects the convergence of {\sc L-FCFW}. 
Morever, we will see in the next section that {\sc L-KM} achieves 
exactly the same sequence of iterates as the above (rather unwieldly) proposal.

\section{Duality between L-KM and L-FCFW (Section \ref{sec:convergence})}

\begin{lemma}
Only vertices in \(\mathcal{A}^{(i)}\) can have positive convex multipliers in the convex decomposition of \(w^{(i)}\), i.e., if we write \(w^{(i)} = \sum_{v \in \mathcal{V}^{(i-1)}}\lambda_{v}^{(i)}v\) such that \(0 \leq \lambda_{v} \leq 1\) for any \(v \in \mathcal{V}^{(i-1)}\), then \(\lambda_{v}^{(i)} = 0\) for any \(v \in \mathcal{V}^{(i-1)} \setminus \mathcal{A}^{(i)}\).
\label{w_in_A}
\end{lemma}
\begin{proof}
By the definition of \(\mathcal{A}^{(i)}\), we have
\begin{align}
\mathop{\mathop{\mathrm{conv}}}(\mathcal{A}^{(i)}) &= \mathop{\mathop{\mathrm{conv}}}(\{v \in \mathcal{V}^{(i-1)}\mid v^{\top}x^{(i)} = {w^{(i)}}^{\top}x^{(i)}\}) \nonumber \\
&= \mathop{\mathop{\mathrm{conv}}}(\{v \in \mathcal{V}^{(i-1)}\mid v^{\top}x^{(i)} = \max_{w \mathop{\in} \mathop{\mathop{\mathrm{conv}}}(\mathcal{V}^{(i-1)})}w^{\top}x^{(i)}\}) \label{maximizer2}  \\
&= {\arg\max}_{w \mathop{\in} \mathop{\mathop{\mathrm{conv}}}(\mathcal{V}^{(i-1)})}w^{\top}x^{(i)}.  \nonumber
\end{align}
Then
\begin{align}
0 &= (w^{(i)} - w^{(i)})^{\top}x^{(i)}    \nonumber  \\
&= (w^{(i)} - \sum_{v \mathop{\in} \mathcal{V}^{(i-1)}}\lambda_{v}^{(i)}v) \\
&= \sum_{v \mathop{\in} \mathcal{V}^{(i-1)} \setminus \mathcal{A}^{i}} \lambda_{v}^{(i)}[(w^{(i)})^{\top}x^{(i)} - v^{\top}x^{(i)}].  &\triangleright\ v^{\top}x^{(i)} = {w^{(i)}}^{\top}x^{(i)},\ \mathop{\forall} v \in \mathcal{A}^{(i)}  \nonumber
\end{align}
By the definition of \(\mathcal{A}^{(i)}\), we have \(v^{\top}x^{(i)} - w^{(i)}x^{(i)} < 0\) for any \(v \in \mathcal{V}^{(i-1)} \setminus \mathcal{A}^{(i)}\). Thus \(\lambda_{v}^{(i)} = 0\) for any \(v \in \mathcal{V}^{(i-1)} \setminus \mathcal{A}^{(i)}\).
\end{proof}

\textbf{Proof of \cref{thm:parallelism}.}
\begin{proof}
We prove by induction. When \(i = 1\), \(\mathcal{V}^{(0)}\) will naturally refer to the same set of points in {\sc L-KM} and {\sc L-FCFW}. By \cref{lemma:uniqueness}, we have \(x^{(1)}\) is the unique solution to \(g + f_{(1)}\). Note \(g^*\) is strongly convex given \(g\) is smooth (\cref{lemma:smoothness_convexity}), we have \(w^{(1)}\) is the unique solution to \(\max_{w \mathop{\in} \mathop{\mathrm{conv}}(\mathcal{V}^{(0)})}-g^*(-w)\). Let \(\mathcal{V} = \mathcal{V}^{(0)}\) in \cref{thm:primal-from-dual}, we have that \(x^{(1)}  = -\nabla g^*(-w^{(1)})\) is the unique minimizer of \(g + f_{(1)}\). So \(x^{(1)}\) in the two algorithms match. Also note that \(w^{(1)}\) solves \(\max_{w \mathop{\in} \mathop{\mathrm{conv}}(\mathcal{V}^{(0)})}-g^*(-w)\), we have \(w^{(1)}\) maximizes \(w^{\top}x^{(1)}\) for all \(w \in \mathop{\mathrm{conv}}(\mathcal{V}^{(i-1)})\) by the first order optimality condition, which gives \({w^{(i)}}^{\top}x^{(i)} = f_{(i)}(x^{(i)})\). Thus \(\mathcal{A}^{(1)}\), \(\mathcal{V}^{(1)}\) match consequently. By strong duality in \cref{thm:strong-duality-subproblems}, we have \(d^{(1)}\) matches in the two algorithms. Note \(g^*\) is strongly convex, which gives the uniqueness of \(w^{(1)}\). By \cref{thm:primal-from-dual}, \(\nabla g(w^{(1)})\) solves the primal subproblem, so \(x^{(1)}\) match in the two algorithms by the uniqueness of \(x^{(1)}\).

Suppose that the theorem holds for \(i = i_0\), in particular, the \(\mathcal{V}^{(i_0)}\) match in the two algorithms. Then for \(i = i_0 + 1\), we can use the same argument as in the previous paragraph by substituting 0 with \(i_0\) and 1 with \(i_0+1\), and show that all the statements hold for \(i = i_0 + 1\). Note that by \cref{w_in_A}, \(\mathcal{A}^{(i)}\) satisfies the condition in \cref{condition_for_B} of {\sc L-FCFW}. Thus this theorem is valid.
\end{proof}

\section{Duality between OSM and L-FCFW (Section \ref{sec:convergence})}
\label{app:duality_osm}

\begin{theorem}
\label{thm:parallelism_osm}
If \(g\) is smooth and strongly convex and in \cref{alg:fcfw} we choose \(\mathcal B^{(i)} = \mathcal{V}^{(i-1)}\),
then
\ben
\item The primal iterates $x^{(i)}$ of \cref{alg:osm} and \cref{alg:fcfw} match.
\item The set  \(\mathcal{V}^{(i)}\) used at each iteration of \cref{alg:osm} and \cref{alg:fcfw} match.
\item The upper and lower bounds $p^{(i)}$ and $d^{(i)}$ of \cref{alg:osm} and \cref{alg:fcfw} match.
\een
\end{theorem}
The proof of \cref{thm:parallelism_osm} is similar to the proof of \cref{thm:parallelism}.

\section{Definition of Diameter and Pyramid Width}
\label{app:p_width}
\textbf{Diameter.} The diameter of a set \(\mathcal{P} \subseteq \mathbb{R}^{n}\) is defined as \begin{equation}
\mathop{\mathrm{Diam}}(\mathcal{P}) \overset{\Delta}{=} \max_{v,\ w \mathop{\in} \mathcal{P}} \|v-w\|_2.
\end{equation}

\textbf{Directional Width.} Given a direction \(x \in \mathbb{R}^{n}\), the directional width of a set \(\mathcal{P} \subseteq \mathbb{R}^{n}\) with respect to \(x\) is defined as
\begin{equation}
\mathop{\mathrm{dirW}}(\mathcal{P},\ x) \overset{\Delta}{=} \max_{v,\ w \mathop{\in} \mathcal{P}} {(v - w)}^{\top}\frac{x}{\|x\|_2}.
\end{equation}

Pyramid directional width and pyramid width are defined by Lacoste-Julien and Jaggi in \cite{Lacoste2015} for a finite sets of vectors \(\mathcal{V} \subseteq \mathbb{R}^{n}\). Here we extend the definition of pyramid width to a polytope \(\mathcal{P} = \mathop{\mathrm{conv}}(V)\), and it should be easy to see that the two definitions are essentially the same.

\textbf{Pyramid Directional Width.} Let \(\mathcal{V} \subseteq \mathbb{R}^{n}\) be a finite set of vectors in \(\mathbb{R}^{n}\). The pyramid directional width of \(\mathcal{V}\) with respect to a direction \(x\) and a base point \(w \in \mathop{\mathrm{conv}}(\mathcal{V})\) is defined as
\begin{equation}
\mathop{\mathrm{PdirW}}(\mathcal{V},\ x,\ w) \overset{\Delta}{=} \min_{A \mathop{\in} \mathcal{A}(w)}\mathop{\mathrm{dirW}}(A \cup \{\mathop{v}(\mathcal{V},\ x)\},\ x),
\end{equation}
where \(\mathcal{A}(w) \overset{\Delta}{=} \{A \subseteq \mathcal{V}\mid \text{the convex multipliers are non-zero for all}\ v \in A\ \text{in the decomposition of}\ w\}\) and \(\mathop{v}(\mathcal{V},\ x)\) is a vector in \(\arg\max_{v \in \mathcal{V}} v^{\top}x\). The pyramid directional width got its name because the set \(A \cup \{\mathop{v}(\mathcal{V},\ x)\}\) has the shape of a pyramid with \(A\) being the base and \(v(\mathcal{V},\ x)\) being the summit.

\textbf{Pyramid Width.}
The pyramid width of \(\mathcal{P}\) is defined as
\begin{equation}
\mathop{\mathrm{PWidth}}(\mathcal{P}) \overset{\Delta}{=} \min_{\mathcal{K} \mathop{\in} \mathop{\mathrm{face}(\mathcal{P})}}\min_{x \mathop{\in} \mathop{\mathrm{cone}}(\mathcal{K}-w) \setminus \{0\},\ w \in {\mathcal{K}}}\mathop{\mathrm{PdirW}}(\mathcal{K}\cap \mathop{\mathrm{vert}(\mathcal{P})},\ x,\ w),
\end{equation}
where \(\mathop{\mathrm{face}(\mathcal{P})}\) stands for the faces of \(\mathcal{P}\) and \(\mathop{\mathrm{cone}}(\mathcal{K}-w)\) is equivalent to the set of vectors pointing inwards \(\mathcal{K}\).

\end{document}


\maketitle

\section{A detailed example}
Here we include some equations and theorem-like environments to show
how these are labeled in a supplement and can be referenced from the
main text.
Consider the following equation:
\begin{equation}
  \label{eq:suppa}
  a^2 + b^2 = c^2.
\end{equation}
You can also reference equations such as \cref{eq:matrices,eq:bb} 
from the main article in this supplement.

\lipsum[100-101]

\begin{theorem}
  An example theorem.
\end{theorem}

\lipsum[102]
 
\begin{lemma}
  An example lemma.
\end{lemma}

\lipsum[103-105]

Here is an example citation: \cite{KoMa14}.

\section[Proof of Thm]{Proof of \cref{thm:bigthm}}
\label{sec:proof}

\lipsum[106-112]

\section{Additional experimental results}
\Cref{tab:foo} shows additional
supporting evidence. 

\begin{table}[htbp]
{\footnotesize
  \caption{Example table}  \label{tab:foo}
\begin{center}
  \begin{tabular}{|c|c|c|} \hline
   Species & \bf Mean & \bf Std.~Dev. \\ \hline
    1 & 3.4 & 1.2 \\
    2 & 5.4 & 0.6 \\ \hline
  \end{tabular}
\end{center}
}
\end{table}

\bibliographystyle{siamplain}
\bibliography{references}